\documentclass[11pt]{amsart}

\usepackage{amssymb,amsmath}
\usepackage{amsthm,enumerate}

\newtheorem{thm}[equation]{Theorem}
\newtheorem{lem}[equation]{Lemma}
\newtheorem{prop}[equation]{Proposition}

\newtheorem{defi}[equation]{Definition}
\newtheorem{rem}[equation]{Remark}
\numberwithin{equation}{section}

\newcommand{\R}{\mathbb{R}}
\newcommand{\N}{\mathbb{N}}

\newdimen\AAdi%
\newbox\AAbo%
\def\AAk#1#2{\setbox\AAbo=\hbox{#2}\AAdi=\wd\AAbo\kern#1\AAdi{}}%

\newcommand{\ent}{\operatorname{Ent}}

\begin{document}

\title[Hamilton Jacobi equations on metric spaces]{Hamilton Jacobi equations on metric spaces and transport entropy inequalities}

\author{N. Gozlan, C. Roberto, P-M. Samson}

\date{\today}

\address{Universit\'e Paris Est Marne la Vall\'ee - Laboratoire d'Analyse et de Math\'e\-matiques Appliqu\'ees (UMR CNRS 8050), 5 bd Descartes, 77454 Marne la Vall\'ee Cedex 2, France}

\address{Universit\'e Paris Ouest Nanterre la D\'efense, MODAL'X, EA 3454, 200 avenue de la R\'epublique 92000 Nanterre, France}

\email{nathael.gozlan@univ-mlv.fr, \hspace{-0,25cm} croberto@math.cnrs.fr, \hspace{-0,25cm} paul-marie.samson@univ-mlv.fr}

\thanks{The authors were partially supported by the ``Agence Nationale de la Recherche'' through the grants ANR 2011 BS01 007 01 and  ANR 10 LABX-58; the second author was partially supported by the European Research Council through the ``Advanced Grant'' PTRELSS 228032.}

\keywords{Transport inequalities, Hamilton-Jacobi equations, logarithmic-Sobolev inequalities, metric spaces}
\subjclass{60E15, 32F32 and 26D10}

\begin{abstract}
We prove an Hopf-Lax-Oleinik formula for the solutions of some Hamilton-Jacobi equations on a general metric space.
As a first consequence, we show in full generality that the log-Sobolev inequality is equivalent to an hypercontractivity property of the Hamilton-Jacobi semi-group.  
As a second consequence, we prove that Talagrand's transport-entropy inequalities in metric space are characterized in terms of log-Sobolev inequalities restricted to the class of $c$-convex functions.
\end{abstract}

\maketitle

%\vspace{1 cm}

\section{Introduction}
Let $L:\R^m\to \R$ be a convex function with super linear growth, in the sense that $L(h)/\|h\|\to \infty$, when $\|h\|\to \infty$, where $\|\cdot\|$ is any norm on $\R^m$. It is well known that if $f$ is some Lipschitz function on $\R^m$, the function $Q_tf$ defined by 
\begin{equation}\label{HLO}
Q_tf(x)=\inf_{y\in \R^m}\left\{f(y)+tL((x-y)/t)\right\},\quad t\geq 0, x\in \R^m,
\end{equation}
is a solution, in different weak senses, of the following Hamilton-Jacobi equation  
\begin{equation}\label{HJ classical}
\partial_t u(t,x)=-L^*(\partial_{x}u(t,x))
\end{equation}
with initial condition $u(0,x)=f(x)$, where $L^*(v)=\sup_{u\in \R^m}\{u\cdot v -L(u)\}$ is the Fenchel-Legendre transform of $L$ (see for instance \cite{Evans-book}). It can be shown, for example, that the function $(t,x) \mapsto Q_tf(x)$ is almost everywhere differentiable in $(0,\infty)\times \R^m$ and that \eqref{HJ classical} is verified at every such point of differentiability (see e.g \cite[Chapter 3]{Evans-book}). Formula \eqref{HLO} is usually referred to as the Hopf-Lax-Oleinik formula for Hamilton-Jacobi equations.

\pagebreak
The objective of this paper is twofold: 
\begin{itemize}
\item[(i)] generalize the Hopf-Lax-Oleinik (HLO) formula to a class of Hamilton-Jacobi equations in a metric space framework;
\item[(ii)] use this aforementioned HLO formula to establish different connections between logarithmic Sobolev type inequalities and transport-entropy inequalities.
\end{itemize}

\subsection{General framework.} 

In this section we give the general setting of this article.

\subsubsection{Assumptions on the space} In all the paper, $(X,d)$ will be a complete and separable metric space in which closed balls are compact. This latter assumption could be removed at the expense of additional standard technicalities. We will sometimes assume that $(X,d)$ is a \emph{geodesic space}, meaning that for every two points $x,y \in X$ there is at least one curve $(\gamma_t)_{t\in [0,1]}$ with $\gamma_0=x$, $\gamma_1=y$ and such that $d(\gamma_s,\gamma_t)=|t-s|d(x,y)$ for all $s,t \in [0,1].$ Such a curve is called a geodesic between $x$ and $y$.
 
\subsubsection{The sup and inf convolution ``semigroups".}In all the paper, $\alpha:\R^+\to\R^+$ will be an increasing convex function of class $\mathcal{C}^1$ such that $\alpha(0)=0.$ If $f:X\to\R$ is a bounded function, we define for all $t>0$ the functions $P_{t}f$ and $Q_{t}f$ as follows:
\begin{equation}\label{Pt}
P_{t}f(x)=\sup_{y\in X}\left\{f(y)-t\alpha\left(\frac{d(x,y)}{t}\right)\right\},\qquad \forall x\in X,
\end{equation}
and
\begin{equation}\label{Qt}
Q_{t}f(x)=\inf_{y\in X}\left\{f(y)+t\alpha\left(\frac{d(x,y)}{t}\right)\right\},\qquad \forall x\in X.
\end{equation}
The operators $P_{t}$ and $Q_{t}$ are connected by the following simple relation
$$Q_{t}f= -P_{t}(-f).$$
When the space $(X,d)$ is geodesic, the families of operators $\{Q_t\}_{t>0}$ and $\{P_t\}_{t>0}$ form non-linear semigroups acting on bounded functions: 
$$Q_{t+s}f=Q_t\left(Q_sf\right)\qquad \text{and}\qquad P_{t+s}f=P_t\left(P_s f\right),\qquad \forall t,s>0,$$
for all bounded function $f:X\to\R.$ When $(X,d)$ is not geodesic, only half of this property is preserved:
$$Q_{t+s}f\leq Q_t\left(Q_sf\right)\qquad \text{and}\qquad P_{t+s}f\geq P_t\left(P_s f\right),\qquad \forall t,s>0.$$

Now we present our main results.

\subsection{An Hopf-Lax-Oleinik formula on a metric space.}
Our objective is to show that the Hamilton-Jacobi equation \eqref{HJ classical} is still verified by $Q_tf$ in the metric space framework introduced above. To that purpose we first need to give a meaning to the state space partial derivative $\partial_x$ in this context. 

We will adopt the following classical measurements $|\nabla^+ f|(x)$ and $|\nabla ^- f|(x)$ of the local slope of a function $f:X\to\R$ around $x\in X$ defined by
\begin{equation}\label{nabla pm}
|\nabla^+ f|(x)=\limsup_{y\to x} \frac{[f(y)-f(x)]_+}{d(x,y)}, \quad |\nabla^- f|(x)=\limsup_{y\to x} \frac{[f(y)-f(x)]_-}{d(x,y)},
\end{equation}
(by convention, we set $|\nabla^\pm f|(x)=0$, if $x$ is an isolated point in $X$). 

If $f$ is locally Lipschitz, then $|\nabla^{\pm} f|(x)$ are finite for every $x\in X$. Moreover, if $f$ is Lipschitz continuous with Lipschitz constant denoted by $\mathrm{Lip}(f)$, then $|\nabla^{\pm}f|(x)\leq \mathrm{Lip}(f)$ for all $x\in X.$ Finally, when $X$ is a Riemannian manifold and $f$ is differentiable at $x$, it is not difficult to check that $|\nabla^\pm f|(x)$ is equal to the norm of the vector $\nabla f(x) \in T_x X$ (the tangent space at $x$).

One of our main result is the following theorem.

\begin{thm}\label{HJ}
If $f:X\to \R$ is an upper semicontinuous function bounded from above, then the following Hamilton-Jacobi differential inequalities hold
\begin{equation}\label{eq HJ}
\frac{d}{dt_+} P_tf(x) \geq \alpha^*\left( |\nabla ^+ P_tf|(x)\right)\qquad \forall t>0,\quad \forall x\in X,
\end{equation}
and 
$$\frac{d}{dt_-} P_tf(x) \geq \alpha^*\left( |\nabla ^- P_tf|(x)\right)\qquad \forall t>0,\quad \forall x\in X,
$$
where $\alpha^*(u)=\sup_{h\geq 0}\left\{ hu-\alpha(h)\right\}$, $u\geq 0$, and where $d/dt_+$ and $d/dt_-$ denote respectively the right and left time derivatives.\\
Moreover, when the space $(X,d)$ is geodesic, it holds
\begin{equation}\label{eq HJ bis}
\frac{d}{dt_+} P_tf(x) = \alpha^*\left( |\nabla ^+ P_tf|(x)\right)\qquad \forall t>0,\quad \forall x\in X.\end{equation}
\end{thm}

The interesting feature of Theorem \ref{HJ} is that there is no measure theory in its formulation: the conclusion holds for \emph{all} $t>0$ and \emph{all} $x\in X$. Theorem \ref{HJ} extends previous results by Lott and Villani \cite{LV07,Villani-book}, where \eqref{eq HJ bis} was obtained on compact \emph{measured} geodesic spaces $(X,d,\mu)$ provided the measure $\mu$ verifies some additional assumptions. More precisely, it is proved in \cite{LV07} that if $\mu$ verifies a doubling condition together with a local Poincar\'e inequality, then \eqref{eq HJ bis} holds true, for all $t$ and for all $x$ outside a set $N_t$ of $\mu$ measure $0$. Under the geometric assumption that $(X,d)$ is finite dimensional with Aleksandrov curvature bounded below, Lott and Villani obtained the validity of \eqref{eq HJ bis} for all $t$ and $x$. In \cite[Theorem 22.46]{Villani-book}, Villani proves \eqref{eq HJ bis} for all $t$ and $x$ on a Riemannian manifold.

We indicate that, during the preparation of this work, we learned that Theorem \ref{HJ} has also been obtained by Ambrosio, Gigli and Savar\'e in their recent paper \cite{AGS12} (see also \cite{AGS12bis}), with a very similar proof. Let us underline that the inequality
\begin{equation}\label{AGS}
\frac{d}{dt_{+}} Q_{t}f(x) \leq -\alpha^*\left(|\nabla^- Q_t f|(x)\right),
\end{equation}
which is equivalent to \eqref{eq HJ}, is an important ingredient in their study of gradient flows of entropic functionals over general metric spaces.  
The main source of inspiration of the present paper is the seminal work by Bobkov, Gentil and Ledoux \cite{BGL01} establishing the equivalence between the logarithmic Sobolev inequality and hypercontractivity properties of Hamilton-Jacobi solutions.

The main tool in the proof of Theorem \ref{HJ} is the following result of independent interest.
\begin{thm}\label{HJ 1}
Let $f:X\to \R$ be an upper semicontinuous function bounded from above. For all $t>0$ and $x\in X$, denote by $m(t,x)$ the set of points where the supremum \eqref{Pt} defining $P_tf(x)$ is reached: 
$$m(t,x)=\left\{\bar{y}\in X : P_tf(x)=f(\bar{y})-t\alpha\left(\frac{d(x,\bar{y})}{t}\right)\right\}.$$
These sets are always non empty and compact and it holds
$$\frac{d}{dt_{+}} P_{t}f(x)=\beta\left(\frac{1}{t}\max_{\bar{y}\in m(t,x)} d(x,\bar{y})\right),\qquad \forall t>0,\quad \forall x\in X$$
and 
$$\frac{d}{dt_{-}} P_{t}f(x)=\beta\left(\frac{1}{t}\min_{\bar{y}\in m(t,x)} d(x,\bar{y})\right),\qquad \forall t>0,\quad \forall x\in X,$$
where $\beta(h)=h\alpha'(h)-\alpha(h)$, $h\geq 0$.
\end{thm}

\subsection{Hypercontractivity of $Q_t$ and the log-Sobolev inequality} \label{sec:hyp}
Let $\mu$ be a Borel probability measure on $X$. Recall that the entropy functional $\ent_\mu(\,\cdot\,)$ is defined by
$$\ent_\mu(g) = \int g \log\left(\frac{g}{\int g\,d\mu}\right)\,d\mu,\qquad \forall g>0.$$
In order to introduce the log-Sobolev inequality, and for technical reasons, define, for $r>0$,
$$
\mathrm{Lip}(f,r) = \sup_{\genfrac{}{}{0pt}{}{x,y:}{d(x,y)\leq r}} \frac{|f(y)-d(x)|}{d(x,y)}
$$
and observe that the usual Lipschitz constant is 
$\mathrm{Lip}(f)=\sup_r \mathrm{Lip}(f,r)$. 
Then, we denote by $\mathcal{F}_\alpha$ the set of bounded functions $f \colon  X \to \mathbb{R}$
such that $\mathrm{Lip}(f,r)<\infty$  for some  $r>0$ and
$$
\mathrm{Lip}(f) \leq \lim_{h\to \infty}\frac{\alpha(h)}{h}
$$
(observe that if $\alpha(h)/h \to \infty$ when $h\to\infty$, this last condition is empty).

The probability measure $\mu$ is said to satisfy the \emph{modified log-Sobolev inequality minus} ${\bf LSI}_\alpha^-(C)$ for some $C>0$ if
\begin{equation} \tag{${\bf LSI}_\alpha^-(C)$}
\ent_\mu(e^f)\leq C\int  \alpha^*(|\nabla^-f|) e^f\,d\mu \qquad \forall f \in \mathcal{F}_\alpha .
\end{equation}
%for all bounded continuous function $f:X \to \mathbb{R}$ such that 

In particular,  when $\alpha(h)= h^p/p$, $h\geq 0$, with $p>1$, it holds $\alpha^*(h)=h^q/q$, $h\geq 0$ with $1/p+1/q=1$. In this case, we write ${\bf LSI}^-_{q}$ for ${\bf LSI}_\alpha^-$. If $X$ is a Riemannian manifold and $\mu$ is absolutely continuous with respect to the volume element, the inequality ${\bf LSI}^-_{2}$ is the usual logarithmic Sobolev inequality introduced by Gross \cite{Gr75}.

Following Bobkov, Gentil and Ledoux \cite{BGL01} we relate ${\bf LSI}_\alpha^-(C)$ to hypercontractivity properties of the family of operators $\{Q_t\}_{t>0}$. To perform the proof, we need to make some restrictions on the function $\alpha.$ We will say that $\alpha$ verifies the \emph{$\Delta_2$-condition} \cite{RR} if there is some positive constant $K$ such that
$$\alpha(2x)\leq K\alpha(x),\qquad \forall x\geq0.$$
\begin{thm}\label{thm-hyper}
Suppose that $\alpha$ verifies the $\Delta_2$-condition. Then the exponents $r_\alpha\leq p_\alpha$ defined by 
$$r_\alpha = \inf_{x>0} \frac{x\alpha'(x)}{\alpha(x)} \geq 1\qquad\text{and}\qquad 1<p_\alpha=\sup_{x>0} \frac{x\alpha'(x)}{\alpha(x)}$$
are both finite. Moreover,
the measure $\mu$ satisfies ${\bf LSI}_\alpha^-(C)$ if and only if for all $t>0$, for all $t_o\leq C(p_\alpha -1)$ and for all bounded continuous function $f:X \to \mathbb{R}$,
\begin{eqnarray}\label{hypercont}
\left\|e^{Q_tf}\right\|_{k(t)} \leq \left\|  e^f\right\|_{k(0)},
\end{eqnarray}
with $$k(t)=\left\{\begin{array}{ll} \left(1+\frac{C^{-1}(t-t_o)}{p_\alpha-1}\right)^{p_\alpha-1}\mathbf{1}_{t\leq t_o} + \left(1+\frac{C^{-1}(t-t_o)}{r_\alpha-1}\right)^{r_\alpha-1}\mathbf{1}_{t> t_o}& \text{ if } r_{\alpha}>1 \\ \min\left(1; \left(1+\frac{C^{-1}(t-t_o)}{p_\alpha-1}\right)^{p_\alpha-1}\right) & \text{ if } r_{\alpha}=1\end{array}\right.,$$
where 
$\|g\|_k= \left(\int |g|^k d\mu\right)^{1/k}$ for $k\neq0$ and $\|g\|_{0}= \exp\left( \int \log g \,d\mu\right)$.
\end{thm}
%\begin{thm}\label{prop-hyper} Let $p>1$ and $q> 1$ with  $1/p+1/q=1$. The measure $\mu$ satisfies ${\bf LSI}_q^-(C)$ if and only if for all $t>0$, for all $a\geq 0$ and for all bounded continuous function $f:X \to \mathbb{R}$,
%\begin{eqnarray}\label{hypercont}
%\left\|e^{Q_tf}\right\|_{k(t)} \leq \left\|  e^f\right\|_{k(0)},
%\end{eqnarray}
%with $k(t)=\left(\frac{C^{-1}t+a}{p-1}\right)^{p-1}$, $\|g\|_k= \left(\int |g|^k d\mu\right)^{1/k}$ for $k\neq0$ and $\|g\|_{0}= \exp\left( \int \log g \,d\mu\right)$.
%\end{thm}
Our proof follows the line of \cite{BGL01}. Let us explain in few words how to derive \eqref{hypercont} from ${\bf LSI}^-_\alpha$. Since $Q_tf\to f$ when $t\to 0$, it is enough to show that $H :t\mapsto \log\left\|e^{Q_tf}\right\|_{k(t)}$ is non-increasing. The left derivative of $H$ has an expression involving $\ent_\mu (e^{k(t)Q_t f})$ and $\int \frac{d}{dt_+}Q_tf e^{k(t)Q_t f}\,d\mu$ (see Proposition \ref{deriv}). To bound the first term from above, we apply the inequality ${\bf LSI}^-_\alpha$. To bound the second term, we use the inequality \eqref{AGS} which is precisely in the right direction to prove that the left derivative of $H$ is negative.

\subsection{From log-Sobolev to transport-entropy inequalities}
Following \cite{BGL01,LV07}, a byproduct of the above hypercontractivity result is a metric space extension of Otto-Villani's theorem \cite{OV00} that indicates that log-Sobolev inequalities imply transport-entropy inequalities.

Let $c:X\times X \to \R$ be a continuous function; recall that the optimal transport cost $\mathcal{T}_{c}(\nu_{1},\nu_{2})$ between two Borel probability measures $\nu_{1},\nu_{2} \in \mathcal{P}(X)$ (the set of all Borel probability measures on $X$) is defined by 
$$\mathcal{T}_{c} (\nu_{1},\nu_{2}) =\inf_{\pi \in P(\nu_{1},\nu_{2})} \iint c(x,y)\,\pi(dxdy),$$
where $P(\nu_{1},\nu_{2})$ is the set of all probability measures $\pi$ on $X\times X$ such that $\pi(dx\times X)=\nu_{1}(dx)$ and $\pi(X\times dy)=\nu_{2}(dy).$ 

The probability measure $\mu$ is said to satisfy the \emph{transport-entropy inequality ${\bf T}_c(C)$}, for some $C>0$ if
\begin{equation} \tag{${\bf T}_c(C)$}
\mathcal{T}_c(\mu,\nu)\leq C H(\nu| \mu),  \qquad\forall \nu \in \mathcal{P}(X),
\end{equation}
where 
$$
H(\nu|\mu)= \left\{ 
\begin{array}{ll}
\int \log \frac{d\nu}{d\mu} \,d\nu & \mbox{if } \nu \ll \mu \\
+\infty & \mbox{otherwise }
\end{array}
\right.
$$ 
is the relative entropy of $\nu$ with respect to $\mu$. This class of inequalities was introduced by Marton and Talagrand \cite{Ma86,Ta96}. When $c(x,y)=\alpha(d(x,y))$ we denote the optimal transport cost by $\mathcal{T}_\alpha(\,\cdot\,,\,\cdot\,)$ and the corresponding transport inequality by ${\bf T}_\alpha.$ In the particular case, when $\alpha(x)=x^p/p,$ $p\geq2$ we use the notation $\mathcal{T}_p$ and ${\bf T}_p$.

The first point of the next theorem will appear to be an easy consequence of Theorem \ref{thm-hyper} and of Bobkov and G\"otze dual formulation of the inequality ${\bf T}_\alpha$ (which roughly speaking corresponds to the hypercontractivity with $t_o=C(p_\alpha-1)$ or equivalently $k(0)=0$).

\begin{thm}\label{Otto-Villani}
Suppose that $\alpha$ verifies the $\Delta_2$-condition. If $\mu$ verifies ${\bf LSI}_{\alpha}^-(C)$, then it verifies ${\bf T}_{\alpha}(A)$, with $$A=\max\left( ((p_\alpha-1)C)^{r_\alpha-1};((p_\alpha-1)C)^{p_\alpha-1}\right),$$ where the numbers $r_\alpha,p_\alpha$ are defined in Theorem \ref{hypercont}.
\end{thm}

In a Riemannian framework and for the quadratic function $\alpha(t)=t^2/2$, Theorem \ref{Otto-Villani} was first obtained by Otto and Villani in \cite{OV00}, closely followed by Bobkov, Gentil and Ledoux \cite{BGL01}. Extensions to other functions $\alpha$ were provided in \cite{BGL01, GGM05}. The path space case was treated by Wang in \cite{W04}. In \cite{LV07}, Lott and Villani extended to certain geodesic measured spaces $(X,d,\mu)$ the Hamilton-Jacobi approach of \cite{BGL01} in the quadratic case. They proved Theorem \ref{Otto-Villani} under additional assumptions on $\mu$ (doubling property and local Poincar\'e). Under the same assumptions Balogh, Engoulatov, Hunziker and Maasalo \cite{BEHM09} treated the case of ${\bf LSI}_q^-$ for all $q\leq2$. The first proofs of Otto-Villani theorem valid on any complete separable metric space appeared in \cite{Go09} and \cite{GRS12}. Their common feature is the use of the stability of the log-Sobolev inequality under tensor products of the reference probability measure. In a recent paper \cite{GL12}, Gigli and Ledoux give another quick proof of Otto-Villani theorem on metric spaces. It is based on calculations along gradient flows in the Wasserstein space.

Using some rough properties of the operators $Q_t$, we also provide a metric space generalization of another result by Otto and Villani \cite{OV00} relating transport-entropy inequalities to Poincar\'e inequality.
\begin{prop}\label{TransversPoincare}
Let $\theta:\R^+\to\R^+$ be any function such that $\theta(x)\geq \min(x^2,a^2)$ for some $a>0$. If $\mu$ verifies ${\bf T}_\theta(C)$ for some $C>0$, then it verifies the following Poincar\'e inequality: $$\mathrm{Var}_\mu (f) \leq \frac{C}{2} \int |\nabla^- f|^2\,d\mu,$$
for all bounded function $f$ such that $\mathrm{Lip}(f,r)<\infty$, for some $r>0$.
\end{prop}

\subsection{Transport-entropy inequalities as restricted log-Sobolev inequalities}

A second consequence of  the Hamilton-Jacobi approach on metric spaces is a characterization of transport-entropy inequalities in terms of log-Sobolev inequalities restricted to a certain class of functions depending on the cost function $\alpha$.

To be more precise, let us say that a function $f$ is \emph{$c$-convex} with respect to a cost function $(x,y)\mapsto c(x,y)$ defined on $X\times X$ if there is a function $g:X\to\R\cup\{\pm\infty\}$ such that 
$$f(x)=P_c g(x) = \sup_{y\in X} \{g(y)-c(x,y)\} \in \R\cup\{\pm\infty\},\qquad \forall x\in X.$$
The class of $c$-convex functions is intimately related to optimal-transport, via for instance the Kantorovich duality theorem (see e.g \cite{Villani-book}).

An important case is when $c(x,y)=\frac{1}{2}\|x-y\|_2^2$ on $\R^m$ (see Proposition \ref{examples} below). In this case, a function $f:\R^m \to\R$ is $c$-convex if and only if the function $x\mapsto f(x)+\|x\|_{2}^2/2$ is convex on $\R^m$. If $f$ is of class $\mathcal{C}^2$, this amounts to say that $\mathrm{Hess}\, f \geq -\mathrm{Id}$.\\

In what follows, we consider the cost $c_p(x,y)=d^p(x,y)/p,$ $p\geq 2$. The second main result of this paper is the following
\begin{thm}\label{main result}
Let $\mu$ be a probability measure on a geodesic space $(X,d)$ and $p\geq 2$. The following properties are equivalent:
\begin{enumerate}
\item There is some $C>0$ such that $\mu$ verifies ${\bf T}_p(C)$.\\
\item There is some $D>0$ such that $\mu$ verifies the following $(\tau)$-log-Sobolev inequality: for all bounded continuous $f$ and all $0<\lambda<1/D$, it holds
$$\ent_\mu(e^f)\leq \frac{1}{1-\lambda D} \int (f-Q^\lambda f)e^f\,d\mu,$$
where for all $\lambda>0$, $Q^\lambda f(x)=\inf_{y\in X}\left\{f(y) +\lambda c_{p}(x,y)\right\}.$\\
\item There is some $E>0$ such that $\mu$ verifies the following restricted log-Sobolev inequality: for all $Kc_p$-convex function $f$, with $0<K<1/E$ it holds 
$$\ent_{\mu}(e^f)\leq \frac{\beta_p(u)-1}{pK^{q-1} (1-KEu)}\int  |\nabla^+ f|^{q}e^f\,d\mu,\qquad \forall u\in(1,1/(KE)),$$
where $q=p/(p-1)$ and $\beta_p(u)= \frac u{[u^{1/(p-1)}-1]^{p-1}}$ for all $u>1.$
\end{enumerate}
The optimal constants $C_\mathrm{opt},D_\mathrm{opt},E_\mathrm{opt}$ are related as follows
$$E_\mathrm{opt} \leq D_\mathrm{opt}\leq C_\mathrm{opt}\leq \kappa_p E_\mathrm{opt},$$ where $\kappa_p$ is some universal constant depending only on $p.$ For $p=2$, one can take $\kappa_2 = e^2.$
\end{thm}
Let us make some comments on Theorem \ref{main result}.
\begin{itemize}
\item The implication $(1)\Rightarrow (2)$ is true for any cost function $c$. It was first proved in \cite{GRS11}.
\item In \cite{GRS12}, we proved that (1) is equivalent to (2) for cost functions $c(x,y)=\alpha(d(x,y))$ as soon as $\alpha$ verifies the $\Delta_2$-condition. Our proof (in \cite{GRS12}) makes use of a tensorization technique and is thus rather different from the one presented here.
\item In \cite{GRS11}, we proved that (1) is equivalent to (3) in a framework essentially Euclidean: $X=\R^m$ and $c(x,y)=\frac{1}{2}\|x-y\|_2^2$. 
\end{itemize}
Theorem \ref{main result} thus provides a wide extension of the results in \cite{GRS11} and unifies nicely the results of \cite{GRS11} and \cite{GRS12}. 

Let us mention that Theorem \ref{main result} as stated above is not as general as possible. Indeed, we will see in Section 5 that this equivalence is still true when the space is not geodesic (Theorem \ref{main result improved}). In this more general framework (3) has to be replaced by a slightly weaker version of the restricted log-Sobolev inequality.  The main tool to prove this extension is Theorem \ref{HJ 1}. It would also be possible to consider more general costs of the form $c(x,y)=\alpha(d(x,y))$ with $\alpha$ satisfying the $\Delta_2$-condition but, to avoid some lengthy developpements, this will not be treated here.

We end this introduction with a short roadmap of the paper.
Section 2 is devoted to $c$-convex functions. In particular, we will recall and prove some well known facts about the subdifferential $\partial_cf(x)$ of a $c$-convex function. In Proposition \ref{Gradients comparisons}, we will relate their gradients $|\nabla^\pm f|(x)$ to the minimal or maximal distance between $x$ and the subdifferential $\partial_cf(x).$
Section 3 contains the proof of the HLO formula. In Section 4, we prove the hypercontractivity property of Theorem \ref{hypercont}, and deduce as a corollary the Otto-Villani Theorem \ref{Otto-Villani}. Section 5 contains the proof of an improved version of our main result Theorem \ref{main result}. Finally, the appendix gathers some technical results.

\newpage 

\tableofcontents

\section{About $c$-convex functions}

In this section we introduce the somehow classical notions of $c$-convex (and $c-$concave) functions and of $c$-subdifferential. We will also give several useful facts about these notions. The interested reader may find more results and comments, and some bibliographic notes, in  \cite[Chapter 5]{Villani-book}.

\subsection{Definition of $c$-convex functions and first results}
Let $X,Y$ be two polish spaces and $c:X\times Y\to\R$ be a general cost function and set $\overline{\R}=\R\cup\{\pm \infty\}.$
For any function $f:X\to \overline{\R}$, we define $Q_cf:Y\to \overline{\R}$ by
$$Q_c f(y):= \inf_{x\in X}\{f(x)+c(x,y)\}.$$
%The function $Q_cf$ is thus the best (i.e the biggest) function $g:Y\to \overline{\R}$ such that
%$$g(y)-f(x)\leq c(x,y),\qquad \forall (x,y)\in X\times Y.$$
For any function $g:Y\to \overline{\R}$, we define $P_cg:X\to \overline{\R}$, by
$$P_c g(x):= \sup_{y\in Y}\{g(y)-c(x,y)\}.$$
%The function $P_cg$ is thus the best (i.e the smallest) function $f:X\to \overline{\R}$ such that
%$$g(y)-f(x)\leq c(x,y),\qquad \forall (x,y)\in X\times Y.$$

\begin{defi}[$c$-convex and $c$-concave functions]
A function $f:X\to\overline{\R}$ is said to be \emph{$c$-convex} if there is some function $g:Y\to \overline{\R}$ such that $f=P_c g.$
A function $g:Y\to \overline{\R}$ is said to be \emph{$c$-concave} if there is some function $f:X \to \overline{\R}$ such that $g = Q_c f.$
\end{defi}
In the definition above, we follow the convention of Villani's book for $c$-convex functions \cite{Villani-book}. Other authors as Rachev and R\"uschendorf \cite{RR-book} define $c$-convex functions as those functions $f:X\to\overline{\R}$ such that there is some function $g:Y \to \overline{\R}$ such that $f(x)=\sup_{y\in Y}\{g(y) + c(x,y)\}.$

\begin{prop}
For any function $f:X\to\overline{\R}$, the inequality $P_cQ_cf\leq f$ holds. 
Moreover, $f:X \to \overline{\R}$ is $c$-convex if and only if $P_c Q_cf =f.$
\end{prop}

\begin{proof}
For the first point observe that; for $z=x$,
$$
P_c Q_c f(x)=\sup_{y\in Y} \inf_{z\in X} \{f(z)+c(z,y)-c(x,y)\}\leq f(x) .
$$
Let us prove the second point. Trivially, a function $f$ such that $f=P_cQ_cf$ is $c$-convex.
Conversely, if $f:X\to \overline{\R}$ is $c$-convex, then there is some function $g$ on $Y$ such that 
$f(x)=\sup_{y\in Y}\{g(y)-c(x,y)\}=Q_c g(y)$. Hence $g$ 
verifies $g(y)\leq \inf_{x\in X}\{f(x) + c(x,y)\}.$ Plugging this inequality into $f=P_{c}g$ gives $f\leq P_cQ_cf$.
Since the other direction always holds, the proof is complete.
\end{proof}

Recall that a function $f:\R^m \to \overline{\R}$ is said to be \emph{closed} (see \cite{Rock-book}) if either $f=-\infty$ everywhere or $f$ takes its values in $\R\cup\{+\infty\}$ and is lower semicontinuous. It is said to be \emph{convex} if its epigraph $\{(x,\alpha)\in \R^m\times \R : \alpha \geq f(x)\}$ is a convex subset of $\R^m \times \R.$ Let us denote by $\Gamma (\R^m)$ the set of all closed and convex functions on $\R^m.$

\begin{prop}[Examples]\label{examples}
Assume that $X=Y=\R^m,$ $m\in\N^*$, equipped with its standard Euclidean structure and let $f:\R^m\to \overline{\R}$.
Then,
\begin{enumerate}
\item If $c(x,y)=x\cdot y,$ $f$ is $c$-convex if and only if $f\in \Gamma(\R^m)$.
\item If $c(x,y)=\frac{1}{2}\|x-y\|_2^2$, $f$ is $c$-convex if and only if $f+\|\cdot\|_2^2/2 \in \Gamma(\R^m)$. In particular, if $f:\R^m \to \R$ is of class $\mathcal{C}^2$ then it is $c$-convex if and only if $\mathrm{Hess}\, f (x) \geq -\mathrm{Id},$ for all $x\in \R^m$.
\end{enumerate}
\end{prop}
\proof\ \\
(1) By definition, a function $f$ is $c$-convex for $c(x,y)=x\cdot y$ if and only if $f=h^*$ for some function $h:\R^m \to \overline{\R}$. It is well known (and easy to check) that $h^* \in \Gamma(\R^m)$ for all $h$. Conversely, if $f\in \Gamma(\R^m)$ then $f=f^{**}$ (see e.g \cite{Rock-book}) and so $f$ is $c$-convex.\\
(2) The function $f$ is a $c$-convex function for $c(x,y)=\|x-y\|_2^2/2$ if and only if $f=P_c g$, for some $g:\R^m\to\overline{\R}.$ Since
$$f(x) + \frac{\|x\|_{2}^2}{2}=\sup_{y\in \R^m}\left\{x\cdot y - \left(\frac{\|y\|_{2}^2}{2}-g(y)\right)\right\},$$
the conclusion follows from the first point.
\endproof

\subsection{The $c$-subdifferential of a $c$-convex function}

In this section we define the notion of $c$-subdifferential of a $c$-convex function and derive some facts that will appear to be useful later.

\begin{defi}[$c$-subdifferential]\label{subdiff}
Let $f:X\to\overline{\R}$ be a $c$-convex function and $x\in X$; the $c$-subdifferential of $f$ at point $x$ is the set, denoted by $\partial_{c}f(x)\subset Y$, of the points $\bar{y} \in Y$ such that
$$f(z)\geq f(x)+c(x,\bar{y}) - c(z,\bar{y}),\qquad \forall z\in X.$$
\end{defi}

The next lemma gives a characterisation of the $c$-subdifferential.

\begin{lem}\label{attainment}
For all $x\in X$, $\partial_cf(x)$ is the set of points $y\in Y$ achieving the supremum in $f(x)=P_cQ_c f(x)$. More precisely, 
$$
\partial_c f(x) = \{ y\in Y : f(x)= Q_cf(y) - c(x,y)\}.
$$
More generally, if $f=P_c g$, for some function $g:Y\to \overline{\R}$, then 
$$
\{y\in Y : f(x)=g(y) - c(x,y)\} \subset \partial_c f(x).
$$
\end{lem}

\begin{proof} 
The first part of the lemma is simple and left to the reader. Let us prove the second part.
Since $f(x)=\sup_{y\in Y}\{g(y) - c(x,y)\}$, $x\in X$, we have $g\leq Q_c f$. So if,
$f(x)=g(\bar{y})-c(x,\bar{y})$ then $f(x)\leq Q_cf(\bar{y})-c(x,\bar{y}) \leq f(z) +c(z,\bar{y}) -c(x,\bar{y})$, for all $z\in X$ which proves that $\bar{y}\in \partial_cf(x).$
\end{proof}

\begin{lem}\label{not empty}
Suppose that the function $c:X\times Y\to \R$ is continuous and bounded from below and that, for all $x\in X$, the level sets $\{y\in Y; c(x,y) \leq r\}$, $r\in \R$, are compact. If $f:X\to \R\cup\{-\infty\}$ is a $c$-convex function bounded from above, then $\partial_cf(x)\neq \emptyset$ for all $x\in X.$
\end{lem}

\begin{proof}
The function $Q_cf$ is an infimum of continuous functions on $Y$, so it is upper semicontinuous on $Y$. For all $x\in X$, the function $\varphi_x: y\mapsto Q_cf(y)-c(x,y)$ is thus upper semicontinuous on $Y$. Since $f$ is bounded from above and $c$ from below, the function $\varphi_x$ is bounded from above. Finally if $y\in \{\varphi_x\geq r\}$ then $c(x,y)\leq \sup f + \inf_{z} c(x,z) - r$. Hence $\{\varphi_x\geq r\}$ is compact. From this follows that $\varphi_x$ achieves its supremum at some point $\bar{y}$ which, according to Lemma \ref{attainment}, necessarily belongs to $\partial_cf(x)$.
\end{proof}

For a better understanding of the notion, in the next lemma we express the $c$-subdifferential of a $c$-convex function $f$ in term of its gradient in some simple cases.

\begin{lem}\label{bysarko1} 
Suppose that $X=Y=\R^m$ and that $c(x,y)=L(x-y)$ where $L:\R^m\to\R^+$ is a differentiable convex function
 with superlinear growth, i.e $L(x)/\|x\|\to +\infty$ when $x\to \infty$, where $\|\cdot\|$ denotes any norm on $\R^m$.
Let $f$ be a $c$-convex function bounded from above differentiable at some point $x$. Then 
$$
\partial_c f(x)= \{x-\nabla (L^*)(-\nabla f(x))\},
$$
where $L^*(y)=\sup_{x\in \R^m}\{x\cdot y -L(y)\}$ is the Fenchel-Legendre transform of $L$. 
\end{lem}

We recall that if $L$ is strictly convex and has a superlinear growth, then its Fenchel-Legendre transform is differentiable everywhere \cite{Rock-book}. 
Lemma \ref{bysarko1} is well known. However, for the sake of completeness, we will recall its proof in the appendix.

\subsection{Comparisons of gradients}
In this last section, as in the rest of the paper, we will assume that $(X,d)$ is a complete separable metric space in which closed balls are compact.
We take $Y=X$ and we consider a cost function $c$ on $X \times X$ of the form $$c(x,y)=\alpha (d(x,y)),$$
where $\alpha:\R^+\to\R^+$ is an increasing convex function of class $\mathcal{C}^1$ such that $\alpha(0)=0$.

%Recall that $(X,d)$ is \emph{geodesic} if any two points of $X$ can be joined by a unit speed geodesic: for all $x,y\in X$, there is a path $(x_s)_{s\in [0,1]}$ with $x_0=x$ and $x_1=y$ such that $d(x_s,x_t)=|s-t|d(x,y).$

%For all locally Lipschitz function $f:X\to \R$ one defines the following subgradients
%$$|\nabla^{-} f|(x)=\limsup_{y\to x} \frac{[f(y)-f(x)]_-}{d(y,x)}\qquad\text{and}\qquad |\nabla^{+} f|(x)=\limsup_{y\to x} \frac{[f(y)-f(x)]_+}{d(y,x)}.$$
If $f:X\to\R$ is $c$-convex for the cost $c(x,y)=\alpha(d(x,y))$, we introduce the following quantities
$$|\nabla_{c}^-f|(x)=\alpha'\left( \inf_{\bar{y}\in \partial_{c}f(x)} d(x,\bar{y})\right)\qquad \text{and}\qquad|\nabla_{c}^+f|(x)=\alpha'\left( \sup_{\bar{y}\in \partial_{c}f(x)} d(x,\bar{y})\right).$$
The following proposition compares $|\nabla^\pm_{c}f|$ to $|\nabla^\pm f|$ defined in \eqref{nabla pm}.

\begin{prop}\label{Gradients comparisons}
Let $f:X\to\R$ be a $c$-convex function for the cost $c(x,y)=\alpha(d(x,y))$. Suppose that $f=P_{c}g$ for some upper semicontinuous function $g:X\to \R$ bounded from above and consider for all $x\in X$ the set $m(x)$ defined by $m(x)=\{ y \in X : f(x)=g(y)-\alpha(d(x,y))\}.$ 
\begin{enumerate}
\item The following inequalities hold
$$ |\nabla^+ f|(x)\leq \alpha'(\max_{\bar{y} \in m(x)} d(x,\bar{y}))\leq |\nabla_{c}^+f|(x).$$
%\stackrel{(\ast)}{\leq} |\nabla^+ f|(x),$$
\item If $(X,d)$ is a geodesic space, then
$$ |\nabla^+ f|(x)= \alpha'(\max_{\bar{y} \in m(x)} d(x,\bar{y}))= |\nabla_{c}^+f|(x).$$
\item The following inequalities hold
$$|\nabla^-f|(x)\leq |\nabla_{c}^-f|(x)\leq \alpha' (\min_{\bar{y}\in m(x)} d(x,\bar{y})).$$
\end{enumerate}
\end{prop}

\begin{rem}
We do not know if there is equality in (3) when the space is geodesic.
\end{rem}

\begin{proof}[Proof of Proposition \ref{Gradients comparisons}]\ 
(1) First observe that, since $f=P_c g$ with $g$ bounded above, $f$ is locally Lipschitz 
(see \cite[Lemma 3.8]{GRS12}), so that $|\nabla^+f|$ is finite everywhere.  
The second inequality is an immediate consequence of the definition of $|\nabla_{c}^+f|(x)$ and the fact that, according to Lemma \ref{attainment}, $m(x)\subset \partial_{c}f(x)$.
Let us prove the first inequality. Let $(x_n)_{n\in \N}$ be a sequence of points converging to $x$, with $x_n\neq x$ for all $n$. For all $n$, fix $y_n\in m(x_n)$. It holds
\begin{align*}
f(x_n)-f(x)&\leq g(y_n)-\alpha(d(x_n,y_n))-\left(g(y_n)-\alpha(d(x,y_n))\right)\\
&\leq d(x,x_n)\alpha'\left(\max(d(x_n,y_n);d(x,y_n))\right),
\end{align*}
where the last inequality follows from  the mean value theorem, the triangle inequality, the non-negativity and the monotonicity of $\alpha'$. Since the function $t\mapsto [t]_+$ is non-decreasing, we get
$$\frac{[f(x_n)-f(x)]_+}{d(x_n,x)}\leq \alpha'(\max(d(x_n,y_n);d(x,y_n))).$$
So letting $n\to\infty$,
\begin{align*}
\limsup_{n\to\infty}\frac{[f(x_n)-f(x)]_+}{d(x_{n},x)}&\leq\alpha'\left(\limsup_{n\to \infty}d(x,y_n)\right)\\&= \alpha'\left(\max\{d(x,\bar{y}) : \bar{y} \text{ limit point of } (y_n)_{n\in \N}\}\right).\\
& \leq \alpha'\left( \max_{\bar{y} \in m(x)} d(x,\bar{y})\right),
\end{align*}
where the last inequality comes from Lemma \ref{Kuratowski} bellow.

\noindent(2) To prove the second point it is enough to show that 
$|\nabla_{c} ^+f|(x) \leq |\nabla ^+ f|(x)$ for all $x\in X$.
Let $\bar{y}\in \partial_cf(x)$. According to the definition of the $c$-subdifferential, 
$$
f(z)-f(x)\geq \alpha(d(x,\bar{y}))-\alpha(d(z,\bar{y})),\qquad \forall z\in X.
$$
From the definition of $|\nabla^+ f|(x)$, it follows that 
$$
|\nabla^+ f|(x)\geq \limsup_{z \to x} \frac{ \alpha(d(x,\bar{y}))-\alpha(d(z,\bar{y}))}{d(x,z)}.
$$
Let $(z_t)_{t\in [0,1]}$ be a geodesic connecting $x$ to $\bar{y}$, it holds $d(x,z_t)=td(x,\bar y)$, $d(z_t,\bar y)=(1-t)d(x,\bar y)$ and therefore 
$$
|\nabla^+ f|(x)\geq \limsup_{t \to 0} \frac{ \alpha(d(x,\bar{y}))-\alpha((1-t)d(x,\bar y))}{td(x,\bar y)}=\alpha'(d(x,\bar y)).
$$
Optimizing over all $\bar y \in \partial_{c}f(x)$ completes the proof.

\noindent(3) Let $(x_n)_{n\in \N}$ be a sequence of points converging to $x$, with $x_n\neq x$ for all $n$. 
If $\bar{y}\in \partial_cf(x)$, then it holds
\begin{align*}
f(x_{n})-f(x)&\geq \alpha\left(d(x,\bar{y})\right)-\alpha\left(d(x_n,\bar{y})\right)\\
&\geq -{d(x,x_{n})}\alpha'\left(\max(d(x_{n},\bar{y});d(x,\bar{y}))\right),
\end{align*}
where the second inequality follows from the mean value theorem and the triangle inequality.
Since the function $t\mapsto[t]_{-}$ is non-increasing, it holds
$$\limsup_{n\to+\infty} \frac{[f(x_{n})-f(x)]_{-}}{d(x,z_{n})}\leq \alpha'\left(d(x,\bar{y})\right).$$
Optimizing over all $\bar{y}\in \partial_cf(x)$ leads to the first bound in (3). As above, the second inequality in $(3)$ is an immediate consequence of the definition of $|\nabla^-_c f|(x)$ together with the fact that, according to Lemma \ref{attainment}, $m(x)\subset \partial_{c}f(x)$. This achieves the proof.
\end{proof}

During the proof we have used the following simple lemma whose proof can be found in the appendix.

\begin{lem}\label{Kuratowski}
Let $X$ be a complete separable metric space with compact balls and $g:X\to \R$ be an upper semicontinuous function bounded from above. 
Define, for all $x \in X$, $P_{t}g(x)=\sup_{y\in X}\left\{ g(y) -t\alpha\left(\frac{d(x,y)}{t}\right) \right\}$
and $m(t,x)$ as the set of points $y \in X$ where this supremum is reached. Then,
\begin{enumerate}
\item The set $m(t,x)$ is a non empty compact set of $X$.
\item Let $x_{n}\to x\in X$ and $t_{n}\to t>0$ be two converging sequences and consider a sequence $(y_{n})_{n\in \N}$ such that $y_{n}\in m(t_{n},x_{n})$ for all $n$. Then $(y_{n})_{n\in \N}$ is bounded and all its limit points belong to $m(t,x).$
\end{enumerate}
\end{lem}

%\begin{lem}\label{bysarko2}Suppose that $X=Y$ is a length space. Let $d$ denote the geodesic distance. Assume that $c(x,y)=\alpha(d(x,y))$ where $\alpha:\R\to\R^+$ is a smooth  function.
%Let $f$ be a $c$-convex function bounded from above. Assume that  $f$ is locally Lipschitz at point $x$ so that $|\nabla^+ f|(x)$ is well defined. Then  for all $\bar y\in\partial_c f(x)$, one has 
%$$ \alpha'(d(x,\bar y))\leq  |\nabla^+ f|(x).$$
%and therefore if $\alpha$ is convex  with superlinear growth,
%$$ \beta(d(x,\bar y))\leq   \alpha^*(|\nabla^+ f|(x)),$$
%with $\beta(h)=h\alpha'(h)-\alpha(h), h\geq 0$.
%\end{lem}

\section{Proof of the Hamilton-Jacobi equations}
This part is devoted to the proof of Theorem \ref{HJ} and \ref{HJ 1}.
%For all function $f:X\to \R$ upper semicontinuous and bounded from above, we define 
%$$P_{t} f(x)=\sup_{{y\in X}}\left\{f(y)-t\alpha\left(\frac{d(x,y)}t\right)\right\},\qquad \forall t>0,\quad \forall x\in X,$$
%and, as in Lemma \ref{Kuratowski} above, we denote by $m(t,x)\subset X$ the set of points $\bar{y}$ where the supremum is reached.   

\proof[Proof of Theorem \ref{HJ 1}]
According to Lemma \ref{Kuratowski}, $m(t,x)$ is a non empty compact set of $X$. 
We treat the case of the right derivative; the other case is completely analogous.
Let $t>0$, $x\in X$ and $(h_{n})_{{n\in \N}}$ a sequence of positive numbers converging to $0$. For all $n\in \N$, we consider $z_{n}\in m(t+h_{n},x).$
Then, 
\begin{align*}
\frac{1}{h_{n}}\left(P_{t+h_{n}}f(x)-P_{t}f(x)\right)
&\leq 
\frac{1}{h_{n}}\left[ f(z_{n})-(t+h_n)\alpha\left(\frac{d(x,z_n)}{t+h_n}\right)-\left(f(z_{n})-t\alpha\left(\frac{d(x,z_n)}t\right)\right)\right]\\
&= 
\frac{1}{h_{n}}\left[t\alpha\left(\frac{d(x,z_n)}t\right) -(t+h_n)\alpha\left(\frac{d(x,z_n)}{t+h_n}\right)\right] .
\end{align*}
Define $D= \limsup_{k\rightarrow \infty} d(x,z_k)$ and take  $\varepsilon>0$. For all $n$ large enough, 
$$d(x,z_n)\leq D+\varepsilon .$$
For all $h\geq 0$, all $t >0$, by the convexity assumption on $\alpha$, the map 
$$d\mapsto t\alpha\left( \frac d t \right)-(t+h)\alpha\left( \frac d {t +h}\right)$$ is non-decreasing. 
Hence
\begin{align*}
\limsup_{n\to\infty} 
\frac{1}{h_{n}} & \left[ t\alpha \left( \frac{d(x,z_n)}t\right) 
-(t+h_n)\alpha\left(\frac{d(x,z_n)}{t+h_n}\right) \right] \\
& \quad \leq 
\lim_{n\to\infty} \frac{1}{h_{n}} \left[ t\alpha\left(\frac{D+\varepsilon}t\right) -(t+h_n)\alpha\left(\frac{D+\varepsilon }{t+h_n}\right)\right] 
=
\beta \left(\frac{D+\varepsilon}t \right)
\end{align*}
where we recall that $\beta(h)=h\alpha'(h)-\alpha(h)$, $h \geq 0$.
Since $\alpha$ is of class ${\mathcal C}^1$, as $\varepsilon$ goes to $0$ we get 
$$
\limsup_{n\to+\infty}\frac{1}{h_{n}}\left(P_{t+h_{n}}f(x)-P_{t}f(x)\right)\leq \beta \left(\frac{D}t \right).
$$
Applying Lemma \ref{Kuratowski}, it is not difficult to check that
$$D=\limsup_{n\to\infty} d(x,z_n)=\max\{d(x,\bar{z}) : \bar{z}\text{ limit point of } (z_n)_{n\in \N}\}\leq \max_{\bar{y}\in m(t,x)} d(x,\bar{y}).$$
The conditions on $\alpha$ ensure that $\beta $ is non-decreasing and therefore 
\begin{equation}\label{liminf}
 \limsup_{n\to+\infty}\frac{1}{h_{n}}\left(P_{t+h_{n}}f(x)-P_{t}f(x)\right)\leq \beta \left(\frac{\max_{\bar{y}\in m(t,x)} d(x,\bar{y})}t \right).
 \end{equation}
Analogously, if $\bar{y} \in m(t,x)$ then  
\begin{align*}
\frac{1}{h_{n}}\left(P_{t+h_{n}}f(x)-P_{t}f(x)\right)&\geq \frac{1}{h_{n}}  \left(t\alpha\left(\frac{d(x,\bar{y})}t\right) -(t+h_n)\alpha\left(\frac{d(x,\bar{y})}{t+h_n}\right)\right)
\end{align*}
So, letting $n$ go to $\infty$, and optimizing over $\bar{y}$ yields
\begin{equation}\label{limsup}
\liminf_{n\to \infty} \frac{1}{h_n}\left(P_{t+h_{n}}f(x)-P_{t}f(x)\right)\geq\beta \left(\frac{\max_{\bar{y}\in m(t,x)} d(x,\bar{y})}t \right).
\end{equation}
We conclude from \eqref{liminf} and \eqref{limsup} that $$\lim_{n\to\infty} \frac 1{h_n}\left(P_{t+h_{n}}f(x)-P_{t}f(x)\right) = \beta \left(\frac{\max_{\bar{y}\in m(t,x)} d(x,\bar{y})}t \right).$$ This completes the proof of proposition \ref{HJ 1}.
\endproof

\proof[Proof of Theorem \ref{HJ}.]
According to Theorem \ref{HJ 1}, $$\frac{d}{dt_+} P_t f(x) = \beta\left(\frac{\max_{\bar{y} \in m(t,x)}d(x,\bar{y})}{t}\right),$$
with $\beta(u)=u\alpha'(u)-\alpha(u)$, for all $u\geq0.$ By definition of the $c$-convexity, the function $x\mapsto P_tf(x)$ is $c$-convex for the cost $c(x,y)=t\alpha\left(\frac{d(x,y)}{t}\right)$. Applying the point (1) of Proposition \ref{Gradients comparisons}, it holds
$$|\nabla^+ P_tf|(x)\leq \alpha'\left(\frac{\max_{\bar{y} \in m(t,x)}d(x,\bar{y})}{t}\right).$$
Observing that $\beta(u)=\alpha^*(\alpha'(u))$ gives the result. According to point (3) of Proposition \ref{Gradients comparisons}, equality holds in the geodesic case. The proof of the inequality involving the left derivative of $P_tf$ is similar.
\endproof

%\begin{prop}
%If $(X,d)$ is geodesic and $f:X\to \R$ is a lower semicontinuous function bounded from below, then
%$$\frac{d}{dt_+} Q_tf(x) = -\alpha^*( |\nabla ^- Q_tf|(x)),\qquad \forall t>0,\quad \forall x\in X.$$
%\end{prop}
%\proof
%Let $\bar{y}\in m(t,x)$, then for all $z\in X$, we have
%\begin{align*}
%Q_tf(z)&\leq f(\bar{y})+t\alpha\left(\frac{d(z,\bar{y})}t\right)\\
%&\leq Q_tf(x)+t\alpha\left(\frac{d(z,\bar{y})}t\right)-t\alpha\left(\frac{d(x,\bar{y})}t\right).
%\end{align*}
%Let $(z_s)_{s\in [0,1]}$ be a unit speed geodesic joining $x$ to $\bar{y}$ ; since the function $t\mapsto [t]_-$ is non-increasing, we get
%\begin{align*}
%\limsup_{z\to x} \frac{[Q_tf(z)-Q_tf(x)]_-}{d(z,x)}&\geq  \limsup_{z\to x} \frac t{d(z,x)}\left[\alpha\left(\frac{d(z,\bar{y})}t\right)-\alpha\left(\frac{d(x,\bar{y})}t\right)\right]_-\\
%&\geq \lim_{s\to 0} \frac t{d(z_s,x)}\left[\alpha\left(\frac{d(z_s,\bar{y})}t\right)-\alpha\left(\frac{d(x,\bar{y})}t\right)\right]_-\\
%&= \lim_{s\to 0} \frac t{sd(x,\bar{y})}\left[\alpha\left(\frac{(1-s)d(x,\bar{y})}t\right)-\alpha\left(\frac{d(x,\bar{y})}t\right)\right]_-\\
%&=\alpha'\left(\frac{d(x,\bar{y}}t\right).
%\end{align*}
%So, according to Proposition \ref{HJ 1},
%$$\alpha^*( |\nabla^- Q_t f|(x))\geq \beta\left(\frac{\max_{\bar{y} \in m(t,x)} d(x,\bar{y})}t\right)=-\frac{d}{dt_+}Q_tf(x).$$
%This inequality together with \eqref{eq HJ 2} completes the proof.
%\endproof

\section{log-Sobolev inequality and hypercontractivity  on a metric space}\label{sectionlogsob}
In this section, following \cite{BGL01}, we show that log-Sobolev inequalities on metric spaces are equivalent to some hypercontractivity property of the ``semigroup" $Q_t$. 
The proof of Theorem \ref{thm-hyper} relies on the differentiation of the left hand side of \eqref{hypercont}. To that purpose, we use the next technical proposition whose proof is postponed to the appendix.

%Recall the definition of $\mathcal{F}_\alpha$ given in Section \ref{sec:hyp}.
%It will be convenient to consider the class of functions $\mathcal{F}_{\alpha}$ defined as follows:\\
%- if $\alpha(h)/h \to \infty$ when $h\to \infty$, $\mathcal{F}_{\alpha}$ is just the set of all bounded continuous functions.\\
%- if $\alpha(h)/h\to \ell \in \R^+$, when $h\to \infty$, $\mathcal{F}_{\alpha}$ is the set of bounded Lipschitz functions $f$ with $\mathrm{Lip}(f)<\ell.$

\begin{prop}\label{deriv}
Let $f$ be a bounded and continuous function on $X$ and $k:(a,b)\to (0,+\infty)$ be a function of class $\mathcal{C}^1$ defined on an open interval $(a,b)\subset (0,\infty)$ and such that $k'(t)\neq 0$ for all $t$. 
Define
$$
H(t)=\frac{1}{k(t)}\log\left(\int e^{k(t)Q_{t}f}\,d\mu \right)
\quad
\mbox{and} 
\quad
K(t)=\frac{1}{k(t)}\log\left(\int e^{k(t)P_{t}f}\,d\mu \right),\qquad  t\in (a,b) .
$$
The functions $H$ and $K$ are continuous and differentiable on the right and on the left on $(a,b)$.
Moreover, for all $t\in (a,b)$, it holds
$$\frac{dH}{dt_+ } (t)=\frac{k'(t)}{k(t)^2} \frac{1}{\int e^{k(t)Q_{t}f}\,d\mu}\left[\ent_{\mu}\left(  e^{k(t)Q_{t}f}\right)+ \frac{k(t)^2}{k'(t)} \int \left(\frac{d}{dt_+ } Q_{t}f\right) e^{k(t)Q_{t}f}\,d\mu\right].$$
The same formula holds for $dH/dt_-$, $dK/dt_+$ and $dK/dt_-$ (replacing $Q_t$ by $P_t$).
%Suppose that $a=0$ and $f\in \mathcal{F}_{\alpha}$, then
%$$H(0^+):=\lim_{t\to 0^+} H(t)=\left\{\begin{array}{ll}\frac{1}{k(0)}\log\int e^{k(0)f}\,d\mu & \text{ if } k(0)\neq 0 \\ \int f\,d\mu & \text{ if } k(0)=0
%\end{array}\right..$$ 
%\noindent Moreover, if $k$ is of class ${\mathcal C}^1$ near $0$ and $k(0),k'(0)\neq0$ then 
%\begin{multline*}
%\limsup_{t\to 0^+}  \frac{H(t)-H(0^+)} t 
%\geq  \frac{k'(0)}{k(0)^2 \int e^{k(0)f}\,d\mu} \times \\
%\left[\ent_{\mu}\left(  e^{k(0)f}\right)
%-  \frac{k(0)^2}{k'(0)}\limsup_{t\to 0} \int   \frac{f-Q_tf}t e^{k(0)f}\,d\mu \right].
%\end{multline*}
\end{prop}

\begin{proof}[Proof of Theorem \ref{thm-hyper}.]
Let us first show that the log-Sobolev inequality implies the hypercontractivity property:
\begin{equation}\label{hypercont-proof}
\left\|e^{Q_tf}\right\|_{k(t)}\leq \left\|e^f\right\|_{k(0)},
\end{equation}
for all bounded continuous function $f:X\to \R,$ with 
\begin{equation}\label{k(t)}
k(t)=\left(1+\frac{C^{-1}(t-t_o)}{p_\alpha-1}\right)^{p_\alpha-1}\mathbf{1}_{t\leq t_o} + \left(1+\frac{C^{-1}(t-t_o)}{r_\alpha-1}\right)^{r_\alpha-1}\mathbf{1}_{t> t_o},
\end{equation}
with the convention that $k(t)=\min\left(1;\left(1+\frac{C^{-1}(t-t_o)}{p_\alpha-1}\right)^{p_\alpha-1} \right),$ if $r_\alpha=1.$
The exponents $r_\alpha$ and $p_\alpha$ have the following property (see \cite[proof of Lemma A.3]{GRS12}):
\begin{align*}
\alpha^*(sx)&\leq s^{\frac{p_\alpha}{p_\alpha-1}}\alpha^*(x),\qquad \forall x\geq0,\forall s\in[0,1]\\
\alpha^*(sx)&\leq s^{\frac{r_\alpha}{r_\alpha-1}}\alpha^*(x),\qquad \forall x\geq0,\forall s>1.
\end{align*}
Let $H(t)=\log \left\|e^{Q_tf}\right\|_{k(t)}$, with $f:X\to \R$ bounded and continuous. 
According to Proposition \ref{deriv}, we have for all $t>0$
$$
\frac{dH}{dt_+ } (t)\leq \frac{k'(t)}{k^2(t)} \frac{1}{\int e^{k(t)Q_{t}f}\,d\mu}\left[\ent_{\mu}\left(  e^{k(t)Q_{t}f}\right)+\frac{k^2(t)}{k'(t)} \int \frac{d}{dt_{+}}Q_{t}f  e^{k(t)Q_{t}f}\,d\mu\right].
$$
Applying ${\bf LSI}_\alpha^-(C)$  to the function $k(t)Q_{t}f$ 
(which belongs to $\mathcal{F}_\alpha$ thanks to Lemma \ref{Falpha} below), it follows that for all $t>0$ (or all $0<t\leq t_o$ if $r_\alpha=1$),
\begin{align*}
\ent_{\mu}\left(e^{k(t)Q_{t}f}\right)&\leq C\int \alpha^*\left(k(t)|\nabla^- Q_{t}f|\right)e^{k(t)Q_{t}f}\,d\mu\\
&\leq C\left(k(t)^{\frac{p_\alpha}{p_\alpha-1}}\mathbf{1}_{t\leq t_o} + k(t)^{\frac{r_\alpha}{r_\alpha-1}}\mathbf{1}_{t> t_o}\right)\int \alpha^*\left(|\nabla^- Q_{t}f|\right)e^{k(t)Q_{t}f}\,d\mu\\
& \leq -C\left(k(t)^{\frac{p_\alpha}{p_\alpha-1}}\mathbf{1}_{t\leq t_o} + k(t)^{\frac{r_\alpha}{r_\alpha-1}}\mathbf{1}_{t> t_o}\right) \int \frac{d}{dt_{+}} Q_{t}fe^{k(t)Q_{t}f}\,d\mu,
\end{align*}
where the last inequality follows from the Hamilton-Jacobi differential inequality \eqref{AGS}. Therefore,
$$\frac{dH}{dt_+ } (t)\leq \frac{1-Ck'(t)\left(k(t)^{\frac{2-p_\alpha}{p_\alpha-1}}\mathbf{1}_{t\leq t_o} + k(t)^{\frac{2-r_\alpha}{r_\alpha-1}}\mathbf{1}_{t> t_o}\right)}{\int e^{k(t)Q_{t}f}\,d\mu} \int \frac{d}{dt_{+}} Q_{t}fe^{k(t)Q_{t}f}\,d\mu =0
$$
where the last equality is a consequence of the very definition of $k$.
Hence $H$ is non-increasing on $(0,+\infty)$ (or on $(0,t_o]$ if $r_\alpha=1$). 
When $\alpha(h)/h \to\infty$, when $h\to\infty,$ then according to point (3) of Proposition \ref{prop Q_{t}} and the dominated convergence theorem, it holds
$$
\log \left\|e^{Q_{t}f}\right\|_{k(t)} = H(t)\leq  \lim_{s \to 0^+} H(s) = \log \left\|e^f\right\|_{k(0)}.
$$
If $\alpha(h)/h \to \ell \in \R^+$, when $h\to\infty$, then according to point (3) of Proposition \ref{prop Q_{t}}, the same conclusion holds if $\mathrm{Lip}(f)<\ell.$
Consider now a bounded continuous function $f : X \to \mathbb{R}$ and fix $\varepsilon \in (0,1)$. 
Thanks to Lemma \ref{Falpha} below,
$\mathrm{Lip}((1-\varepsilon)Q_s f)\leq (1-\varepsilon)\ell$ for all $s>0$.
%
%
%If $\alpha(h)/h \to\infty$ when $h\to\infty$, the proof is finished. If $\alpha(h)/h \to\ell \in \R^+$ then we have proved that \eqref{hypercont-proof} holds for all $f$ such that $\mathrm{Lip}(f)<\ell$. It is not difficult to see that \eqref{hypercont-proof} still holds if $\mathrm{Lip}(f)\leq \ell$ (apply \eqref{hypercont-proof} to $f/(1+\varepsilon)$ with $\varepsilon>0$ going to $0$). Finally, note that if $f$ is bounded continuous, then for all $s>0$, $Q_sf$ is $\ell$-Lipschitz. 
Since $Q_sf\leq f$, we can conclude that
$$
%\left\|e^{Q_{t+s} f}\right\|_{k(t)} \leq 
\left\|e^{Q_t ((1-\varepsilon) Q_s f)}\right\|_{k(t)}\leq \left\|e^{(1-\varepsilon)Q_sf}\right\|_{k(0)}\leq \left\|e^{(1-\varepsilon)f}\right\|_{k(0)}.
$$
Using Lebesgue's Theorem and Lemma \ref{Falpha}, as $\varepsilon \to 0$, we get
$$
\left\|e^{Q_t ( Q_s f)}\right\|_{k(t)}
\leq 
\left\| e^{f}\right\|_{k(0)}.
$$
Since $Q_{t+s}f\leq Q_t(Q_sf)$ and thanks to point (2) of Proposition \ref{prop Q_{t}}, we have $\lim_{s \to 0}  Q_{t+s} f = Q_t f$ so that (using Lebesgue's theorem) the hypercontractivity property  \eqref{hypercont-proof} still holds when $f$ is bounded and continuous, as expected.

\smallskip

Now we prove that if \eqref{hypercont-proof} holds for all bounded continuous $f$ and all $t>0$ with $k$ defined by \eqref{k(t)}, then $\mu$ verifies ${\bf LSI}_\alpha^-(C)$. Observe that in the case $\alpha(h)/h\to \ell\in \R^+$, it is enough to show that ${\bf LSI}_\alpha^-$ holds for functions with $\mathrm{Lip}(f)<\ell.$

Let $H(t)=\log \left\|e^{Q_tf}\right\|_{k(t)}$, for all $t>0$, 
with $f\in \mathcal{F}_\alpha$ and $\mathrm{Lip} (f)< \ell$ when $\alpha(h)/h \to\ell \in \R^+$ as $h\to\infty.$ By assumption, it holds
$$
\limsup_{t\to 0^+} \frac{H(t)-H(0^+)}t\leq 0.
$$
Let us choose $t_o<C(p_\alpha-1)$ in the definition of $k(t)$ so that $k(0)$ and $k'(0)>0$.
It is not difficult to check that
\begin{equation*}
\limsup_{t\to 0^+}  \frac{H(t)-H(0^+)} t =  \frac{k'(0)}{k(0)^2} \frac{\ent_{\mu}\left(  e^{k(0)f}\right)}{\int e^{k(0)f}\,d\mu}
- \frac{1}{k(0)\int e^{k(0)f}\,d\mu} \liminf_{t\to 0^+}\int   \frac{e^{k(t)f}-e^{k(t)Q_tf}}t \,d\mu.
\end{equation*}
According to the mean value theorem, there exists a function $\varphi:(0,\infty)\times X \to \R$ taking values in the interval $[k(t)e^{k(t)Q_{t}f(x)} ; k(t)e^{k(t)f(x)}]$ such that
$$ \frac{e^{k(t)f}-e^{k(t)Q_tf}}t = \frac{f-Q_tf}t \varphi(t,x),\qquad \forall t>0,x\in X.$$
Applying point (4) of Proposition \ref{prop Q_{t}}, we get
\begin{align*} 
\liminf_{t\to 0^+} \int   \frac{e^{k(t)f}-e^{k(t)Q_tf}}t \,d\mu &\leq \limsup_{t\to 0^+} \int   \frac{e^{k(t)f}-e^{k(t)Q_tf}}t \,d\mu\\
&\leq k(0 ) \int \alpha^*(|\nabla^- f|) \,e^{k(0)f}\,d\mu .
\end{align*}
So
$$
\ent_{\mu}\left(  e^{k(0)f}\right)
\leq 
\frac{k(0)^2} {k'(0)}\int  \alpha^*\left(|\nabla^-f|\right)  e^{k(0)f}\,d\mu.
$$
Since $k(0)=\left(1-\frac{C^{-1}t_o}{p_\alpha-1}\right)^{p_\alpha-1}\to 1$ 
and $k(0)^2/k'(0) = C \left(1-\frac{C^{-1}t_o}{p_\alpha-1}\right)^{p_\alpha-2}\to C$, 
when $t_o\to0^+,$ we conclude that ${\bf LSI}_\alpha^-(C)$ holds. This completes the proof.
%The calculation of $H(0^+)$ follows easily from point (3) of 
%Proposition \ref{prop Q_{t}}, a simple Taylor expansion (in the case $k(0)=0$) and the dominated convergence theorem.
%for all $f\in \mathcal{F}_\alpha.$ If $\alpha(h)/h\to \infty$ when $h\to\infty,$ the proof is finished. 
%Suppose that $\alpha(h)/h\to\ell\in \R^+$ when $h\to\infty$ and take a 
%bounded Lipschitz function $f$ with $\mathrm{Lip}(f)\leq \ell.$ 
%Then for all $\varepsilon>0$, $f_\varepsilon:=f/(1+\varepsilon)$ is in $\mathcal{F}_\alpha$. 
%So the log-Sobolev inequality holds for $f_\varepsilon$; letting $\varepsilon$ go to $0$, 
%one easily checks that it also holds for $f$.
\end{proof}

During the proof above, we used the following technical lemma whose proof is postponed to the appendix for the clarity of the exposition.

\begin{lem} \label{Falpha}
Set $\ell = \lim_{h \to \infty}\frac{\alpha(h)}{h} \in \mathbb{R}\cup\{+\infty\}$.
Let $f:X\to \R$ be a bounded and continuous function. 
Then,
\begin{enumerate}
\item For all $t >0$, $Q_t f \in \mathcal{F}_\alpha$ and $\mathrm{Lip}(Q_t f) \leq \ell$.
\item For all $t>0$ and all $x \in X$,  $\lim_{\varepsilon \to 0} Q_t ((1-\varepsilon) f)(x) = Q_t f(x)$.
\end{enumerate}
\end{lem}

We are now in position to derive the Otto-Villani Theorem from Theorem \ref{thm-hyper}. 

Recall that, according to Bobkov and G\"otze characterization \cite{BG99}, $\mu$ verifies the transport-entropy inequality ${\bf T}_\alpha(C)$ if and only if 
\begin{equation}\label{BG}
\int e^{C^{-1}Q_1f}\,d\mu \leq \exp\left(C^{-1}\int f\,d\mu\right),
\end{equation}
for all bounded continuous function $f:X\to \R.$

\proof[Proof of Theorem \ref{Otto-Villani}]
Since $\mu$ verifies ${\bf LSI}_\alpha^{-}(C)$, it verifies the hypercontractivity property \eqref{hypercont} of Theorem \ref{thm-hyper}. Take $t_o=C(p_\alpha-1)$ in the definition of $k(t)$, the hypercontractivity inequality \eqref{hypercont} yields for all bounded continuous function $f$,
$$\int e^{k(t)Q_{t}f} d\mu\leq e^{k(t)\int f d\mu},\qquad \forall t>0.$$
According to \eqref{BG}, this means that $\mu$ verifies the following family of transport-entropy inequalities
$$\mathcal{T}_{\alpha(\,\cdot\, /t)}(\mu,\nu)\leq \frac{1}{tk(t)} H(\nu|\mu),\qquad \forall \nu \in \mathcal{P}(X),$$
where $\alpha(\,\cdot\, /t)$ denotes the function $x\mapsto \alpha(x/t)$. According to \cite[Proof of Lemma A.3]{GRS12}, 
$$\alpha(x) \leq \max(t^{r_\alpha} ; t^{p_\alpha}) \alpha(x/t),\qquad \forall t>0.$$
Therefore, $\mu$ verifies ${\bf T}_\alpha(A)$, with the constant 
$$A=\inf_{t>0} \frac{\max(t^{r_\alpha-1} ; t^{p_\alpha-1})}{k(t)}.$$
Taking $t=C(p_\alpha-1)$ for which $k(t)=1$, we see that $$A\leq \max\left( ((p_\alpha-1)C)^{r_\alpha-1};((p_\alpha-1)C)^{p_\alpha-1}\right),$$ which ends the proof.
\endproof

\proof[Proof of Proposition \ref{TransversPoincare}]
Define for all $t>0$ the operators $$R_t f(x)=\inf_{y\in X}\left\{f(y) + \frac{1}{t} \theta(d(x,y))\right\}\quad\text{and}\quad Q_t f(x)=\inf_{y\in X}\left\{f(y) + \frac{1}{t} d^2(x,y)\right\}$$
According to Bobkov and G\"otze dual formula \eqref{BG} and by homogeneity, it holds for all $t>0$
$$\int e^{C^{-1}tR_t f}\, d\mu \leq e^{C^{-1}t\int f\,d\mu},$$
for all bounded continuous function $f.$
Take a function $f$ such that $|f|\leq M$ and $\mathrm{Lip}(f,r)<\infty$ for some $r>0$. If $d(x,y)\geq a$, and $t\leq a^2/(2M)$, then it holds
$$f(y)+\frac{1}{t}\theta(d(x,y)) \geq -M + \frac{(2M)}{a^2} a^2=M \geq f(x) \geq R_tf(x).$$
It follows that if $t\leq a^2/2M$, then
$$R_t f(x)\geq \inf_{ y : d(x,y)\leq a}\left\{f(y) + \frac{1}{t}d^2(x,y))\right\} \geq Q_tf(x).$$
So the following inequality holds
$$\int e^{C^{-1}tQ_tf}\,d\mu \leq e^{C^{-1}t\int f\,d\mu},\qquad \forall t\leq a^2/(2M).$$
Applying Taylor formula, we see that
$$e^{C^{-1}tQ_t f(x)} = 1 + C^{-1}tQ_tf(x) + \frac{C^{-2}(tQ_tf)^2(x)}{2}e^{\varphi(t,x)},$$
where $|\varphi(t,x)| \leq tM$, for all $t,x$. 
So, for all $t\leq a^2/(2M)$,
$$C^{-1}\int \frac{Q_tf-f}{t}\,d\mu + \frac{C^{-2}}{2}\int (Q_tf)^2(x)e^{\varphi(t,x)}\,\mu(dx)\leq \frac{e^{C^{-1}t\int f\,d\mu}-1 -tC^{-1}\int f\,d\mu}{t^2}.$$
Letting $t$ go to $0$ and using points (3) and (4) of Proposition \ref{prop Q_{t}} together with the dominated convergence theorem yields to
$$-\frac{C^{-1}}{4}\int |\nabla^-f|^2\,d\mu + \frac{C^{-2}}{2} \int f^2\,d\mu \leq \frac{C^{-2}}{2} \left(\int f\,d\mu\right)^2, $$
which is the announced Poincar\'e inequality.
\endproof

\section{Transport-entropy inequalities as restricted log-Sobolev inequalities}\label{transpotlogsob}

In this section, we show that a transport-entropy inequality  can be characterized as a modified log-Sobolev inequality restricted to a class of $c$-convex functions. Actually we will prove the following improved version of Theorem \ref{main result} which holds even if the space is not geodesic.

\begin{thm}\label{main result improved}
Let $\mu$ be a probability measure on $(X,d)$ and $p\geq 2$. Define the function $\beta_p$ as follows:
\begin{equation}\label{beta_p}
\beta_p(u)= \frac u{[u^{1/(p-1)}-1]^{p-1}},\qquad \forall u>1.
\end{equation}
The following properties are equivalent:
\begin{enumerate}
\item There is some $C>0$ such that $\mu$ verifies ${\bf T}_p(C)$.\\
\item There is some $D>0$ such that $\mu$ verifies the following $(\tau)$-log-Sobolev inequality: for all bounded continuous $f$ and all $0<\lambda<1/D$, it holds
$$\ent_\mu(e^f)\leq \frac{1}{1-\lambda D} \int (f-Q^\lambda f)e^f\,d\mu,$$
where for all $\lambda>0$, $Q^\lambda f(x)=\inf_{y\in X}\left\{f(y) +\lambda c_{p}(x,y)\right\}.$\\
\item There is some $E>0$ such that $\mu$ verifies the following restricted log-Sobolev inequality: for all $Kc_p$-convex function $f$, with $0<K<1/E$ it holds 
$$\ent_{\mu}(e^f)\leq \frac{\beta_p(u)-1}{(1-KEu)pK^{q-1}}\int  |\nabla_{Kc_p}^- f|^{q}e^f\, d\mu,\qquad \forall u\in(1,1/(KE))$$
where $q=p/(p-1)$ and $|\nabla_{Kc_p}^-f|(x)= K\left(\inf_{\bar{y}\in \partial_{Kc_p}f(x)} d(x,\bar{y})\right)^{p-1}$ (see Proposition \ref{Gradients comparisons}).
\end{enumerate}
Moreover, when the space $(X,d)$ is geodesic these properties are equivalent to the following
\begin{itemize}
\item[\textit{(3')}] There is some $F>0$ such that $\mu$ verifies the following restricted log-Sobolev inequality: for all $Kc_p$-convex function $f$, with $0<K<1/F$ it holds 
$$\ent_{\mu}(e^f)\leq \frac{\beta_p(u)-1}{(1-KFu)pK^{q-1}}\int  |\nabla^+ f|^{q}e^f\, d\mu,\qquad \forall u\in(1,1/(KF))$$
\end{itemize}
The optimal constants $C_\mathrm{opt},D_\mathrm{opt},E_\mathrm{opt},F_\mathrm{opt}$ are related as follows
$$F_\mathrm{opt} \leq E_\mathrm{opt} \leq D_\mathrm{opt}\leq C_\mathrm{opt}\leq \kappa_p F_\mathrm{opt},$$ where $\kappa_p$ is some universal constant depending only on $p.$ For $p=2$, one can take $\kappa_2 = e^2.$
\end{thm}

\subsection{From transport-entropy inequalities to $(\tau)$-log-Sobolev inequalities}

Let us recall the following proposition from \cite{GRS11} whose proof relies on a simple Jensen argument.
\begin{lem}
If $\mu$ verifies the transport-entropy property $\mathbf{T}_{c}(C)$, for some continuous cost function $c$ on $X^2$, then the following ($\tau$)-log-Sobolev property holds: for all function $f$, for all $0<\lambda<1/C$, 
\begin{eqnarray}\label{taulogsob}
\ent_{\mu}(e^f)\leq \frac{1}{1-\lambda C} \int  (f-Q^\lambda f) e^f\,d\mu,
\end{eqnarray}
where for all $x\in X$,
$Q^\lambda f(x)=\inf\{f(y)+\lambda c(x,y)\}.$
\end{lem}

This proves the step $(1) \Rightarrow (2)$ in Theorem \ref{main result improved}.

\subsection{From transport entropy inequalities to log-Sobolev inequalities for $c_{p}$-convex functions}
The general link between the ($\tau$)-log-Sobolev 
property and the restricted log-Sobolev inequality is the following:  if the function $f$ is $c$-convex then the quantity $f-Q^\lambda f$ in the right-hand side of $\eqref{taulogsob}$ can be bounded by a function of $|\nabla_c^- f|$ (see Lemma \ref{adieupec} below). 

From now on, let us  assume that $c=c_p$ is the cost function defined by: for all $x,y$ in $X$,
$c_p(x,y)=d^p(x,y)/p,$ for some $p>1$.

\begin{lem}\label{adieupec} Let $\lambda>0$. If $f$ is a $Kc_p$-convex function bounded from above, and if $0<K<\lambda$, then for all $x\in X$ and all $\bar y$ in the $Kc_p$-subdifferential of $f$ at point $x$, $\partial_{Kc_p} f(x)$,
$$f(x)-Q^\lambda f(x)\leq K  \left(\beta_p\left( \lambda /K\right) -1\right)c_p(x,\bar y),$$
where $Q^\lambda f(x)=\inf_{y\in X}\{f(y)+\lambda c_{p}(x,y)\}$ and for all $u>1$,
$\beta_p(u)= \frac u{[u^{1/(p-1)}-1]^{p-1}}.$\\
Equivalently, with the notation of Proposition \ref{Gradients comparisons}, 
$$f(x)-Q^\lambda f(x) \leq (\beta_{p}(\lambda/K)-1) \frac{1}{pK^{q-1}} |\nabla^-_{Kc_{p}}f|^q(x),$$
where $q=\frac{p}{p-1}.$
\end{lem}

\proof According to Definition \ref{subdiff} of $\partial_{Kc_p} f(x)$ and using the triangular inequality we get, for all $\bar y\in \partial_{Kc_p} f(x)$
\begin{align*}
 f(x)-Q^\lambda f(x)&=\sup_{z\in X} \{f(x)-f(z)-\lambda c_p(z,x)\}\\ 
 &\leq\sup_{z\in X} \{ Kc_p(z,\bar y)- Kc_p(x,\bar y) -\lambda c_p(z,x)\}\\
 &\leq \sup_{z\in X} \{ Kc_p(z,\bar y) -\lambda c_p(z,x)\}  - Kc_p(x,\bar y)\\
  &\leq \frac 1p\sup_{z\in X}\{ K(d(z,x)+ d(x,\bar y))^p -\lambda d^p(z,x)\}  - Kc_p(x,\bar y)\\
  &\leq \frac 1p\sup_{r\geq 0 }\{ K(r+ d(x,\bar y))^p -\lambda r^p\}  - Kc_p(x,\bar y)\\
  &= Kc_p(x,\bar y)  \left(\beta_p\left( \lambda /K\right) -1\right).
 \end{align*} 
Thus optimizing over all possible $\bar{y} \in \partial_{Kc_{p}}f(x)$ yields to the expected result
\begin{align*}
f(x)-Q^\lambda f(x)&\leq  \left(\beta_p\left( \lambda /K\right) -1\right) \inf_{\bar{y} \in \partial_{Kc_{p}}f(x)}    Kc_p(x,\bar y)
=(\beta_{p}(\lambda/K)-1) \frac{1}{pK^{q-1}} |\nabla^-_{Kc_{p}}f|^q(x) .
\end{align*}
\endproof

From this  lemma  the ($\tau$)-log-Sobolev property \eqref{taulogsob} provides immediately the first part of the following statement by setting $u=\lambda/C$.

\begin{prop}\label{transverslog}
If $\mu$ verifies the $(\tau)$-log-Sobolev \eqref{taulogsob} with the cost $c=c_{p}$, $p>1$, then for all $K\in(0,1/C)$ and all function $f$ bounded from above and $Kc_{p}$-convex, it holds
$$\ent_{\mu}(e^f)\leq \frac{\beta_p(u)-1}{(1-KCu)pK^{q-1}} \int  |\nabla^-_{Kc_{p}}f|^q(x)\,e^{f(x)}\,\mu(dx),\quad \forall u\in(1,1/(KC)).$$
Moreover, when $(X,d)$ is geodesic, the same inequality holds with $|\nabla^+ f|$ instead of $|\nabla_{K_{c_{p}}}^-f|$ in the right-hand side.
\end{prop}

This proves the steps $(2)\Rightarrow (3)$ and $(2)\Rightarrow (3')$ (in the geodesic case) in Theorem \ref{main result improved}.

\proof
Let us justify the statement in the geodesic case. According to Proposition \ref{Gradients comparisons} (applied with the function $\theta(x)=Kx^p/p$), it holds $|\nabla_{Kc_{p}}^- f|\leq |\nabla_{Kc_{p}}^+ f|$ and when the space is geodesic, $|\nabla_{Kc_{p}}^+f |=|\nabla^+f|$, which completes the proof.
\endproof

\subsection{From log-Sobolev inequalities for $c_{p}$-convex functions to transport-entropy inequalities}
In this part we prove that a modified log-Sobolev inequality restricted to the class of $Kc_{p}$-convex functions also implies a transport entropy-inequality. One of the main ingredient of the proof is Theorem \ref{HJ 1}.
\begin{thm}\label{logverstrans}
Let $p\geq 2$. Suppose that for all $K\in( 0,1/C)$ and all $Kc_p$-convex function $f:X\to \R$ bounded from above, it holds 
\begin{equation}\label{eq logverstrans}
\ent_{\mu}(e^f)\leq \frac{\beta_p(u)-1}{(1-KCu)pK^{q-1}} \int  |\nabla^-_{Kc_{p}}f|^q(x)\,e^{f(x)}\,\mu(dx),\quad \forall u\in(1,1/(KC)).
\end{equation}
then $\mu$ verifies the inequality $\mathbf{T}_{p}(\kappa_{p}C)$, where $\kappa_{p}$ is some numerical constant depending only on $p.$
For $p=2$, $\kappa_{2}=e^2.$ Moreover, if the space is geodesic, the same conclusion holds if $|\nabla^-_{Kc_{p}}f|$ is replaced by $|\nabla^+f|$ in the right hand side of \eqref{eq logverstrans}.
\end{thm}
This proves the steps $(3)\Rightarrow (1)$ and $(3')\Rightarrow (1)$ (in the geodesic case) and completes the proof of Theorem \ref{main result improved}.
\proof
For any bounded continuous function $g$, we define the function $P_{t}g$ as follows
$$
P_{t}g(x)=\sup_{y\in X}\left\{ g(y) - \frac{1}{t^{p-1}} c_{p}(x,y)\right\}.
$$
Let $\ell : [a,1]\to (0,+\infty)$ be a decreasing function of class 
$\mathcal{C}^1$ defined on some interval $[a,1]$ with $a>0$ and such that $\ell (1)=0.$
For all bounded continuous $g$ define 
$H_{g}(t)=\frac{C}{\ell(t)} \log\left(\int e^{C^{-1}\ell(t)P_{t}g}\,d\mu\right)$, $t\in [a,1).$ 
If all the $H_{g}$'s were non-decreasing, then it would hold that $H_{g}(a)\leq \lim_{t\to 1^-} H_{g}(t)=\int P_{1}f\,d\mu.$ Since $g\leq P_{a}g$, we would get
$$
\int e^{C^{-1}\ell(a) g}\,d\mu \leq e^{C^{-1}\ell(a) \int P_{1} g\,d\mu}
$$
which in turn, according to Bobkov and G\"otze characterization Theorem, would prove that $\mu$ verifies 
$\mathbf{T}_{p}(C/\ell(a)).$

Hence, our aim is to construct a function $\ell$ such that all the $H_{g}$'s are non-decreasing. 
Set $f_{t}=C^{-1}\ell(t)P_{t}g$.
According to Proposition \ref{deriv}, $H_{g}$ is continuous and differentiable on the right and
$$\frac{d}{dt_{+}} H_{g}(t)=\frac{C\ell'(t)}{\ell^2(t) \int e^{f_{t}}\,d\mu} \left[\ent_{\mu}\left(e^{f_{t}}\right)+\frac{\ell(t)^2}{C\ell'(t)} \int \frac{dP_{t}g }{dt_{+}} e^{f_{t}}\,d\mu\right].$$
Since $\ell'<0$, all we have to  show is that the term into brackets is non-positive. 
For all $t>0$, the function $f_{t}$ is $K(t)c_{p}$-convex, with $K(t)=\frac{\ell(t)}{Ct^{p-1}}$. Hence, 
for all $t$ such that $\ell(t)< t^{p-1}$ and all  $u\in(1,1/(CK(t))$,
$$
\ent_{\mu}(e^{f_{t}})\leq \frac{\beta_p(u)-1}{(1-K(t)Cu)pK(t)^{q-1}}\int |\nabla_{K(t)c_{p}}^-(f_{t})|^q(x)e^{f_{t}(x)}\,\mu(dx) .
$$
Since $f_{t}$ is $K(t)c_{p}$-convex, it follows from Proposition \ref{Gradients comparisons} (applied with  $\alpha(h)=K(t)h^p/p$) that
$$
|\nabla^-_{K(t)c_{p}} f_{t}|(x) = K(t) \left(\min_{\bar{y} \in \partial_{K(t)c_{p}}f_{t}(x)} d(x,\bar{y})\right)^{p-1} \leq K(t) \left(\max_{\bar{y} \in m(t,x)} d(x,\bar{y})\right)^{p-1},
$$
denoting by $m(t,x)$ the set of points $\bar{y}$ where the supremum defining $P_{t}g$ is reached.
As a result, it holds
$$\frac{1}{p K(t)^{q-1}} |\nabla^-_{K(t)c_{p}} f_{t}|^q(x) \leq K(t) \max_{\bar{y}\in m(t,x)} c_{p}(x,\bar{y}).$$
On the other hand, according to Proposition \ref{HJ 1}, 
$$
\frac{dP_{t}g}{dt_{+}}(x)=\frac{p-1}{t^p}\max_{\bar{y}\in m(t,x)} c_{p}(x,\bar{y}).
$$
Therefore
\begin{equation}\label{a changer}
\frac{1}{p K(t)^{q-1}} |\nabla^-_{K(t)c_{p}} f_{t}|^q(x) \leq \frac{K(t)t^p}{(p-1)} \frac{dP_{t}g}{dt_{+}}(x)=\frac{t\ell(t)}{(p-1)C}\frac{dP_{t}g}{dt_{+}}(x).
\end{equation}
So, for all $t>0$ with $\ell(t)<t^{p-1}$ it holds
\begin{multline*}
\left[\ent_{\mu}\left(e^{f_{t}}\right)+\frac{\ell(t)^2}{C\ell'(t)} \int \frac{dP_{t}g }{dt_{+}} e^{f_{t}}\,d\mu\right] 
\leq 
\frac{\ell(t)}{C} \left[\theta_{p}\left(\frac{\ell(t)}{t^{p-1}}\right) \frac{t}{p-1}+\frac{\ell(t)}{\ell'(t)}\right]\int\frac{dP_{t}g }{dt_{+}} e^{f_{t}}\,d\mu,
\end{multline*}
where the function $\theta_{p}$ is defined by $\theta_{p}(x)= \inf_{1<u<1/x}\left\{\frac{\beta_p(u)-1}{1- x u}\right\}$, for  $x<1$. Observe that $\theta_{p}$ is finite on $[0,1[.$ Consider the function 
$$\Psi_{p}(r)=\frac{1}{p-1}\int_{0}^{r} \frac {\theta_p(s)}{s(\theta_p(s)+1)}\,ds,\qquad \forall r\in [0,1].$$
According to Lemma \ref{legueanvichy} below, since $p\geq 2$,  the function $\Psi_{p}$ is well defined, increasing and of class $\mathcal{C}^1$ on $(0,1).$
Define $v(t)=\Psi_{p}^{-1} (-\ln(t)),$ for all $t\in[ a_{p}, 1]$, with $a_{p}= \exp\left(-\Psi_{p}(1)\right)$. The function $v$ is increasing  and $v(t)\in[0,1]$ for all $t\in[ a_{p}, 1]$. Finally, define $\ell_{p}(t) = t^{p-1}v(t)$, for all $t\in [a_{p},1].$ A simple calculation shows that 
$$\theta_{p}\left(\frac{\ell_{p}(t)}{t^{p-1}}\right) \frac{t}{p-1}+\frac{\ell_{p}(t)}{\ell_{p}'(t)} = 0,\qquad \forall t\in (a_{p},1).$$
We conclude that $\mu$ verifies the inequality $\mathbf{T}_{p}$ with the constant $$\frac{C}{\ell_{p}(a_{p})}=C\exp\left(\int_{0}^1\frac {\theta_p(s)}{s(\theta_p(s)+1)}\,ds\right)=C\kappa_p.$$
In the particular case $p=2$, one has $\theta_{2}(x)=\frac{4x}{\left(1-x\right)^2}$, and it is easy to check that $\kappa_{2}=e^{2}$.

It remains to consider the geodesic case. In this case, the inequality \eqref{a changer} is replaced by the equality
$$\frac{1}{p K(t)^{q-1}} |\nabla^+ f_{t}|^q(x) = \frac{K(t)t^p}{(p-1)} \frac{dP_{t}g}{dt_{+}}(x),
$$
and the rest of the proof remains unchanged.
\endproof

\begin{lem}\label{legueanvichy}
The function $s \mapsto \phi(s)=\frac {\theta_p(s)}{s(\theta_p(s)+1)}$ is continuous on $(0,1)$. Moreover,  $\phi(s)$ goes to $1$ as $s$ goes to 1 and 
$$\phi(s)=\frac {p^{p/(p-1)}}{ s^{(p-2)/(p-1)}}(1+\varepsilon(s)),$$ with $\varepsilon(s)\to 0$ as $s\to 0$.
\end{lem}

\proof
After some computations, it is  easily to check that for $s\in(0,1)$, the infimum $\theta_p(s)$ is reached at some unique point $u=u(s)\in (1,1/s)$ such that 
$$ \beta'_p(u)(1-su)+s(\beta_p(u)-1)=0,$$
or equivalently 
$$ u(s)^{p/(p-1)}-\left( u(s)^{1/(p-1)}-1\right)^p=1/s.$$
It follows from this equality that $u(s)$ is continuous on (0,1), $u(s)\to 1$ as $s\to 1$ and $u(s)\to +\infty $ as $s\to 0$.
As a first consequence, $\phi$ is continuous on $(0,1)$.   

By a Taylor expansion at point 0, one has 
\begin{equation*}
\frac{1}{su(s)^{p/(p-1)}}= 1-\left(1-\frac{1}{u(s)^{1/(p-1)}}\right)^p=\frac{p}{u(s)^{1/(p-1)}}(1+\varepsilon(s)),
\end{equation*}
with $\varepsilon(s)\to 0$ as $s\to 0$. It follows that $su(s)\to 1/p$ as $s\to 0$.
From all this observations, we get 
$$\phi(s)=\frac{1-\left(1-u(s)^{-1/(p-1)}\right)^{p-1}}{s\left(1-su(s)\left(1-u(s)^{-1/(p-1)}\right)^{p-1}\right)}=\frac {p^{p/(p-1)}}{ s^{(p-2)/(p-1)}}(1+\varepsilon(s)),$$
with $\varepsilon(s)\to 0$ as $s\to 0$. 
Since $u(s)\to 1$ as $s\to 1$ we  easily get that $\phi(s)\to 1$ as $s\to 1$.
\endproof

\appendix

\section{Proof of Lemma \ref{bysarko1}, Lemma \ref{Kuratowski}, Proposition \ref{deriv} and Lemma \ref{Falpha}}

In this appendix we collect all the technical proofs of Lemmas \ref{bysarko1}, \ref{Kuratowski} and \ref{Falpha} and of Proposition \ref{deriv}.

\proof[Proof of Lemma \ref{bysarko1}]
Let $\bar{y}\in \partial_cf(x)$. According to the definition of the $c$-subdifferential, 
$$f(z)-f(x)\geq L(x-\bar{y})-L(z-\bar{y}),\qquad \forall z\in \R^m.$$
Let $z= x+\varepsilon u$ with $\varepsilon >0$ and $u\in \R^m$. Since $L$ and $f$ are smooth functions at point $x$, we get as $\varepsilon$ tends to $0$, for all $u\in \R^m$,
$$u\cdot\nabla f(x) \geq -u\cdot \nabla L(x-\bar{y}),$$
and therefore $\nabla f(x)=- \nabla L(x-\bar{y}).$ Let $v_o=x-\bar{y}$ and $u_o=\nabla L(v_0)$, by the convexity property of $L$, 
\begin{equation}\label{eq:subgradient}
L(v)\geq L(v_o) +u_o\cdot(v-v_o),\qquad \forall v\in \R^m,
\end{equation}
or equivalently 
$L(v_o)\leq u_o\cdot v_o-L^*(u_o)$. Since $L(v_o)=\sup_{u\in \R^m}\{u\cdot v_o-L^*(u)\}$, it follows that the derivative of $u\mapsto u\cdot v_o - L^*(u)$ vanishes at $u_o$, and so $v_o=\nabla L^*(u_o)$.
Finally, $x-\bar{y}=\nabla L^*(u_o)=\nabla L^*(-\nabla f(x))$, which completes the proof.
\endproof

\proof[Proof of Lemma \ref{Kuratowski}]
(1) The function $h:y\mapsto g(y)-t\alpha\left(d(x,y)/t\right)$ is upper semicontinuous, bounded from above and its level sets $\{h\geq r\}$ $r\in \R$ are compact. It follows that $h$ reaches its supremum and so $m(t,x)=\{h\geq \sup h\}$ is not empty and compact.\\
(2) Let $h_n(y)=g(y)-t_n\alpha\left(\frac{d(x_n,y)}{t_n}\right),$ $y\in X.$ The sequence of functions $h_n$ converges pointwise to the function $h$, and the convergence is uniform on each bounded set. Since $g$ is bounded from above by some constant $r\in \R$, it holds
\begin{equation}\label{eq Kuratowski}
r-t_n\alpha\left(\frac{d(x_n,y_n)}{t_n}\right)\geq g(y_n)-t_n\alpha\left(\frac{d(x_n,y_n)}{t_n}\right)\geq g(y)-t_n\alpha\left(\frac{d(x_n,y)}{t_n}\right),\qquad \forall y\in X.
\end{equation}
Since $(x_n)_{n\in \N}$ is bounded and $\lim_{n\rightarrow \infty}t_n=t>0$, we conclude that $(y_n)_{n\in \N}$ is a bounded sequence. As balls are supposed to be compact, $(y_n)_{n\in \N}$ has converging subsequences. Passing to the limit into the inequality \eqref{eq Kuratowski} along a converging subsequence of $(y_n)_{n\in \N}$, yields to the conclusion that any limit point $\bar{y}$ of $(y_n)_{n\in \N}$ belongs to $m(t,x).$
\endproof

Let us turn to the proof of Proposition \ref{deriv}. The proof requires some regularity properties of $Q_tf$ in the $t$ variable that are gathered in the following proposition. 
%Recall the definition of $\mathcal{F}_\alpha$ given in Section \ref{sec:hyp}.

\begin{prop}\label{prop Q_{t}}
Let $f$ be a bounded lower semicontinuous function on $X$; define for all $t>0$ and $x\in X$ $Q_tf (x) = \inf\left\{f(y)+t\alpha\left(\frac{d(x,y)}{t}\right)\right\}$ and let $m(t,x)$ denote the set of points where this infimum is attained. The following properties hold
\begin{enumerate}
\item For all $x\in X$, 
$$ 
m(t,x)\subset B\left(x, t\alpha^{-1}\left({\mathrm{Osc}(f)}/t\right)\right).
$$
\item For all $t,h>0$,
$$
\frac{1}{h}\sup_{x\in X} |Q_{t+h}f(x)-Q_{t}f(x)|
\leq 
\beta\left(\alpha^{-1}\left({\mathrm{Osc}(f)}/t\right)\right).$$
\item If $\alpha(h)/h\to\infty$, when $h\to \infty$, then for all bounded continuous function $f$ and for all $x\in X$, 
$$
\lim_{t\to 0^+} Q_{t}f(x)= f(x).
$$ 
and
$$
\liminf_{t\rightarrow 0^+} \frac{Q_tf(x)-f(x)}t
\geq 
-\alpha^*(|\nabla^- f|(x)).
$$
 If $\alpha(h)/h\to\ell \in \R^+$, when $h\to \infty$, the same conclusions hold for all function $f$ with $\mathrm{Lip}(f)<\ell.$
\item  Let $\mu$ be a probability measure and $\varphi:(0,+\infty)\times X \to \R$ be such that $|\varphi|\leq M$ for some $M>0$ and $\lim_{t\to0+}\varphi(t,x)=\psi(x)$ for all $x\in X$. If $\alpha(h)/h \to\infty$ when $h\to\infty$ and if $f$ is such that $\mathrm{Lip}(f,r)<+\infty$ for some $r>0$, then
$$
\limsup_{t \to 0} \int \frac{f-Q_t f}{t}\varphi(t,x)\,d\mu \leq \int \alpha^*(|\nabla^- f|(x))\psi(x)\,d\mu.
$$
The same conclusion holds if $\alpha(h)/h \to \ell \in \R^+,$ when $h\to\infty,$ and $\mathrm{Lip}(f)<\ell.$
\end{enumerate}
\end{prop}

\proof[Proof of Proposition \ref{prop Q_{t}}]  
(1) Let $M=\sup (f)$ and $m=\inf(f)$. If $\bar{y}\in m(t,x)$, it holds
$$
m+t\alpha\left(\frac{d(x,\bar{y})}t\right)
\leq 
f(\bar{y})+t\alpha\left(\frac{d(x,\bar{y})}t\right)=Q_{t}f(x)\leq M,$$
which proves the first claim.

\noindent (2) Since $t\mapsto Q_{t}f(x)$ is non-increasing, $|Q_{t+h}f(x)-Q_{t}f(x)|=Q_{t}f(x)-Q_{t+h}f(x).$ If $\bar{y}\in m(t+h,x)$, then
\begin{align*}
\frac 1h \left(Q_{t}f(x)-Q_{t+h}f(x)\right) & \leq \frac 1h \left(t\alpha\left(\frac{d(x,\bar{y})}t\right)-(t+h)\alpha\left(\frac{d(x,\bar{y})}{t+h}\right)\right)
\leq \beta\left(\alpha^{-1}\left({\mathrm{Osc}(f)}/t\right)\right),
\end{align*}
where the last inequality comes from  the mean value theorem, the monotonicity of the function $\beta$ and point (1).

\noindent(3) Let us first assume that $\lim_{h\rightarrow \infty}\alpha(h)/h=+\infty$. In this case, $\lim_{t\rightarrow 0}t\alpha^{-1}\left(\frac{\mathrm{Osc}(f)}t\right)=0$ and so, according to the first point,
$$
\inf_{y\in B\left(x, t\alpha^{-1}\left({\mathrm{Osc}(f)}/t\right)\right)} \{f(y)\}\leq Q_{t}f(x)\leq f(x).
$$
Since $f$ is lower semicontinuous, the limit when $t$ goes to $0$ of the left hand side is greater than or equal to $f(x)$. This guarantees that $\lim_{t\to 0^+} Q_{t}f(x)= f(x)$.
Moreover, for all $\bar y_t\in m(t,x)$, $f(\bar y_t)\leq f(x)$ and therefore
\begin{align} \label{mamamia}
\frac {f(x)-Q_{t}f(x)}t
&=
\frac{f(x)-f(\bar y_t)}{t} -\alpha\left(\frac{d(x,\bar y_t)}t\right)
=
\frac{[f(\bar y_t)-f(x)]_-}{d(x,\bar y_t)}\,\frac{d(x,\bar y_t)}t -\alpha\left(\frac{d(x,\bar y_t)}t\right)
\nonumber \\
&\leq \alpha^*\left(\frac{[f(\bar y_t)-f(x)]_-}{d(x,\bar y_t)}\right).
\end{align}
Arguing as before, we see that $\bar y_t\to x$ as $t\to 0$ so that 
$$
\limsup_{t\to 0^+} \frac {f(x)-Q_{t}f(x)}t\leq \alpha^*\left(|\nabla^-f|(x)\right).
$$
Now let us assume that $\alpha(h)/h\to\ell \in \R^+$ when $h\to\infty$. 
According to what precedes, it is enough to show that there is a constant $r>0$ such that
$$
m(t,x)\subset B(x; rt),\qquad \forall t>0,x\in X.
$$ 
Let $\bar{y} \in m(t,x)$. Then it holds $f(\bar{y})-f(x)+t\alpha\left(d(x,\bar{y})/t\right)\leq 0.$ Since $f$ is assumed to be Lipschitz, we conclude that $\mathrm{Lip}(f) d(x,\bar{y})/t\geq \alpha\left( d(x,\bar{y})/t\right).$ Since $\mathrm{Lip} (f)<\ell = \lim_{h\to+\infty} \alpha(h)/h$, this implies that $d(x,\bar{y})\leq rt$ where $r=\sup\{h : \alpha(h)/h \leq \mathrm{Lip}(f)\}<+\infty$, which proves the claim.

\noindent (4) We already know, by point (3), that $\limsup_{t\to 0^+} \frac {f(x)-Q_{t}f(x)}t\leq \alpha^*\left(|\nabla^-f|(x)\right)$. Hence the result of point (4) will follow from Fatou's Lemma (in its limsup version) as soon as for some $t_0>0$, it holds $\sup_{x} \sup_{t \in (0,t_o)}\frac {f(x)-Q_{t}f(x)}t < \infty$.

Assume first that $\lim_{h \to \infty} \alpha(h)/h=\infty$ and let $r>0$ be such that $\mathrm{Lip}(f,r)< \infty$. Observe that $\lim_{t\rightarrow 0}t\alpha^{-1}\left(\frac{\mathrm{Osc}(f)}t\right)=0$ so that, by point (1), 
there exists $t_o>0$ such that, for all $t \in (0,t_o)$, all $x\in X$ and all $\bar y_t\in m(t,x)$,
$d(x,\bar y_t) \leq r$. Using \eqref{mamamia}, we conclude that
$\sup_{x} \sup_{t \in (0,t_o)}\frac {f(x)-Q_{t}f(x)}t \leq \alpha^*\left(\mathrm{Lip}(f,r)\right) < \infty$.

Assume now that $\alpha(h)/h\to\ell\in \R^+$, when $h\to\infty.$ Then, since $\mathrm{Lip}(f) < \ell$, \eqref{mamamia} implies that
$\sup_{x,t}\frac{f(x)-Q_{t}f(x)}t \leq \alpha^*\left(\mathrm{Lip}(f)\right) < \infty$.
This ends the proof of point (4) and of the proposition.
\endproof

\proof[Proof of Proposition \ref{deriv}] We will prove that $H$ is right differentiable, the proof of the left-differentiability being similar.
By formally differentiating under the sign integral yields for all $t>0$,
\begin{align}\label{dif int}
\frac{dH}{dt_+ }(t)
& =
-\frac{k'(t)}{k(t)^2}\log\left(\int e^{k(t)Q_{t}f}\,d\mu \right) \nonumber\\
& \quad 
+ \frac{1}{k(t)\int e^{k(t)Q_{t}f}\,d\mu} 
\left[\int  k'(t)Q_{t}fe^{k(t)Q_{t}f}\,d\mu  + \int k(t)\frac{d}{dt_+ }Q_{t}fe^{k(t)Q_{t}f}\,d\mu\right],
\end{align}
which easily gives the desired identity. Hence, it remains  to  justify the above calculation.
Define $F(t)=\int e^{k(t) Q_{t}f }\,d\mu.$ To obtain \eqref{dif int}, it is enough to show that $F$ is right differentiable and that
$$
\frac{dF}{dt_+ } (t)=\int k'(t)Q_{t}f e^{k(t)Q_{t}f}\,d\mu + \int k(t) \frac{d}{dt_+ }Q_{t}f e^{k(t) Q_{t}f}\,d\mu.$$
For all $s>0$, 
$\frac{1}{s}\left(F(t+s)-F(t)\right)= \int G_{s} \,d\mu,$
with $G_{s}=\frac{1}{s}\left(e^{k(t+s)Q_{t+s}f} -e^{k(t)Q_{t}f}\right).$
Since $t\mapsto Q_{t}f(x)$ is right differentiable for $t>0$, 
$$
G_{s}(x) \underset{s\to 0}{\longrightarrow} k'(t)Q_{t}f(x) e^{k(t)Q_{t}f(x)} + k(t) \frac{d}{dt_+ }Q_{t}f(x) e^{k(t)Q_{t}f(x)} .%,\qquad \text{when } s\to0.
$$
For a given $t\in (a,b)$, let $\eta_{t}>0$ be any number such that $t+\eta_{t}<b$. Then, using the mean value Theorem together with point (2) of Proposition   \ref{prop Q_{t}}, it is not difficult to prove that  $\sup_{x\in X}\sup_{s\leq \eta_{t}} |G_{s}| (x)<+\infty.$
%
%Moreover, for $0<s\leq 1$, $(t+s)|Q_{t+s}f|\leq (t+1)\sup |f|:=M_{t}$, so applying the mean value theorem, it holds
%\begin{align*}
%|G_{s}|&\leq De^{M_{t}}\frac{1}{s}|(t+s)Q_{t+s}f - tQ_{t}f|\\
%& \leq De^{M_{t}}\left(\mathrm{Osc}f + \sup |f| \right),
%\end{align*}
%where the second inequality follows easily from point (3) of Proposition \ref{prop Q_{t}}.
%So $G_{s}$ is uniformly bounded. 
Applying the dominated convergence theorem completes the proof.
\endproof

\begin{proof}[Proof of Lemma \ref{Falpha}]
Let $f:X\to \R$ be a bounded and continuous function. Fix $t>0$.

\noindent
(1) 
First, following \cite[Lemma 3.8]{GRS12}, we will prove that there exists $r>0$ such that $\mathrm{Lip}(Q_t f,r) < \infty.$
Set $r=t \alpha^{-1}(\mathrm{Osc}(f)/t)$. From point (1) of Proposition \ref{prop Q_{t}}, it holds
%For all $u\in X$ and all $y\in X$ such that $d(y,u)>r$, we have
%$$
%f(y)+t \alpha(d(u,y)/t)>-\|f\|_{\infty}+t\alpha(r/t)=\|f\|_{\infty}.
%$$
%Since $Q_t f\leq \|f\|_{\infty}$, we conclude that 
$$Q_tf(u)=\inf_{d(y,u)\leq r} \left\{f(y)+t\alpha(d(u,y)/t)\right\}, \qquad \forall u \in X.$$
%Fix $u_{o}\in X$ and let $B_{o}$ be the closed ball of center $u_{o}$ and radius $2r$. From the previous computations, it holds that for any $u\in B_{o}$, $Q_tf(u)=\inf_{y\in B_{o}}\left\{f(y)+t\alpha(d(u,y)/t)\right\}$. 
Fix $u,v \in X$ with $d(u,v) \leq r$. Then, given $y_o \in X$ such that $d(v,y_o) \leq r$, it follows from the mean value theorem that
%Furthermore, for all $u,v,y\in B_{o}$, 
\begin{align*}
|t\alpha(d(u,y_o)/t)-t\alpha(d(v,y_o)/t)| 
& \leq  
|d(v,y_o)-d(u,y_o)|\max_{s\in [0,1]} \alpha'([sd(u,y_o)+(1-s)d(v,y_o)]/t)\\
&\leq 
\alpha'(2r/t)d(u,v) .
\end{align*}
Now, let $y_o$ be such that $Q_tf(v) = f(y_o)+t\alpha(d(v,y_o)/t)$ and observe that, thanks to the previous observation,
$d(v,y_0) \leq r$. It follows that (choosing $y=y_o$),
\begin{align*}
Q_t f(u) - Q_t f(v) 
& =
\inf_{y} \left\{f(y)+t\alpha(d(u,y)/t)\right\}  - f(y_o) -t \alpha(d(v,y_o)/t) \\
& \leq
t \alpha(d(u,y_o)/t) - t \alpha(d(v,y_o)/t) \\
& \leq 
\alpha'(2r/t)d(u,v) ,
\end{align*} 
which proves that $\mathrm{Lip}(Q_t f, r)< \infty$.

Now assume that $\alpha(h)/h\to\ell \in \R^+$, when $h\to \infty$ and let us prove that $Q_{t}f$ is $\ell$-Lipschitz. The convexity of $\alpha$ implies that
$$\frac{\alpha(h)}{h}\leq \alpha'(h)\leq \frac{\alpha(2h)-\alpha(h)}{h},\qquad \forall h>0.$$
So $\sup_h \alpha'(h) = \lim_{h\to \infty} \alpha'(h)=\ell$ and it follows that $Q_{t}f$ is $\ell$-Lipschitz as an infimum of $\ell$-Lipschitz functions.

%In order to prove that $(1-\varepsilon)Q_t f \in \mathcal{F}_\alpha$, for $\varepsilon \in (0,1)$, it remains to show that 
%$(1-\varepsilon) \mathrm{Lip}(Q_t f) < \lim_{h \to \infty} \frac{\alpha(h)}{h}$. 
%This will follow from the fact that $\mathrm{Lip}(Q_t f) \leq \lim_{h \to \infty} \frac{\alpha(h)}{h}$ that we will now prove.
%If the latter limit is $+\infty$, there is nothing to prove. Hence, we assume that such a limit is finite and equals $\ell$.
%We claim that $\sup_{x \in \mathbb{R}^+} \alpha'(x) = \ell$. Indeed, we have
%$$
%x \alpha'(x) \leq \int_x^{2x} \alpha'(t) dt \leq \alpha(2x) - \alpha(x)
%$$
%so that $ \alpha'(x) \leq 2 \frac{\alpha(2x)}{2x}-\frac{\alpha(x)}{x}$. Since $x \mapsto \alpha'(x)$
%is non-decreasing, we can conclude that 
%$\sup_{x \in \mathbb{R}^+} \alpha'(x) = \lim_{x \to \infty} \alpha'(x) \leq \ell$.
%On the other hand, since $\alpha(0)=0$, by convexity it holds $\alpha'(x) \geq \alpha(x)/x$, $x>0$.  Hence, $\sup_{x \in \mathbb{R}^+} \alpha'(x) \geq \ell$ and the claim is proven.

%

%Fix $u,v \in X$. Then, for any $y \in X$, as above
%\begin{align*}
%|t\alpha(d(u,y)/t)-t\alpha(d(v,y)/t)| 
%& \leq  
%|d(v,y)-d(u,y)|\max_{s\in [0,1]} \alpha'([sd(u,y)+(1-s)d(v,y)]/t)\\
%&\leq 
%\ell d(u,v) .
%\end{align*}
%Hence the map $u \mapsto f(y) +t\alpha(d(v,y)/t)$ is $\ell$-Lipschitz for any $y$.
%In turn, $u \mapsto Q_tf(u)=\inf_{y \in X} \left\{f(y)+t\alpha(d(u,y)/t)\right\}$ is also $\ell$-Lipschitz as supremum of  $\ell$-Lipschitz functions. This achieves the proof of point (1).

\smallskip
\noindent (2)  
Let $(\lambda_n)_{n \geq 0}$ be a sequence of real numbers converging to $1$.
For any $x \in X$, let $m(t,x)$ be the set of points $y \in X$ such that
$Q_tf(x)=\inf_{z \in X} \{f(z)+t\alpha(d(x,z)/t)\}=f(y)+t\alpha(d(x,y)/t)$.
For any $n$, let $y_n$ be such that $Q_t (\lambda_n f)(x)= \lambda_n f(y_n) + t\alpha(d(x,y_n)/t)$.
We have, for all $z \in X$,
$$
\lambda_n \inf f + t\alpha(d(x,y_n)/t) \leq \lambda_n f(y_n) + t\alpha(d(x,y_n)/t) \leq
\lambda_n  f(z) + t\alpha(d(x,z)/t) .
$$
Since $(\lambda_n)_n$ converges, we deduce that the sequence $(y_n)_n$ is bounded. Let
$y$ be a limit point of a converging subsequence of $(y_n)_n$. Passing to the limit in the latter leads to
$$
f(y) + t\alpha(d(x,y)/t) \leq
f(z) + t\alpha(d(x,z)/t) \qquad \forall z \in X .
$$
Hence, $y \in m(t,x)$. In turn, after easy considerations left to the reader, $Q_t (\lambda_n f)(x)\to Q_t f(x)$, when $n\to\infty$ as expected. The conclusion of point (2) follows and the proof is complete.
\end{proof}

\bibliographystyle{plain}

\end{document}